\documentclass[12pt]{amsart}
\usepackage{fullpage}

\input xy
\xyoption{all}

\newcommand{\ord}{\mbox{ord}}

\newcommand{\id}{\mbox{id}}

\usepackage{color, graphicx}
\usepackage{amsmath}
\usepackage{amssymb}
\usepackage{comment}

\usepackage{charter}
\usepackage[T1]{fontenc}
\usepackage{textcomp}

\usepackage{calrsfs}

\theoremstyle{plain}

\newtheorem*{theorem*}{Theorem}
\newtheorem*{corollary*}{Corollary}
\newtheorem{theorem}{Theorem}[section]
\newtheorem{lemma}[theorem]{\bf Lemma}
\newtheorem{corollary}[theorem]{\bf Corollary}
\newtheorem{proposition}[theorem]{\bf Proposition}

\theoremstyle{definition}

\theoremstyle{remark}
\newtheorem{remark}[theorem]{\bf Remark}

\newtheorem{notation}[theorem]{\bf Notation}
\newtheorem{notationassumptions}[theorem]{\bf Notation and Assumptions}
\newtheorem{assumption}[theorem]{\bf Assumption}
\newtheorem{claim}[theorem]{\bf Claim}

\renewcommand{\AA}{{\mathfrak A}}
\newcommand{\BB}{{\mathfrak B}}
\newcommand{\EE}{{\mathfrak E}}

\newcommand{\calS}{{\mathcal S}}

\newcommand{\calV}{{\mathcal V}}

\newcommand{\F}{{\mathbb F}}
\newcommand{\G}{{\mathbb G}}

\newcommand{\Q}{{\mathbb Q}}

\newcommand{\Z}{{\mathbb Z}}

\newcommand{\Aa}{{\mathfrak A}}
\newcommand{\pp}{{\mathfrak p}}

\newcommand{\Qq}{{\mathfrak Q}}
\newcommand{\qq}{{\mathfrak q}}

\newcommand{\ttt}{{\mathfrak t}}
\newcommand{\aaa}{{\mathfrak a}}

\newcommand{\be}{\begin{enumerate}}
\newcommand{\ee}{\end{enumerate}}

\def\ord{\mathop{\mathrm{ord}}\nolimits}

\usepackage[OT2,OT1]{fontenc}
\newcommand\cyr{%
\renewcommand\rmdefault{wncyr}%
\renewcommand\sfdefault{wncyss}%
\renewcommand\encodingdefault{OT2}%
\normalfont
\selectfont}
\DeclareTextFontCommand{\textcyr}{\cyr}

\begin{document}
\bibliographystyle{plain}%
 \title{On definitions of polynomials over function fields of positive characteristic}%
\author{Alexandra Shlapentokh}%
\thanks{The research for this paper has been partially supported by NSF grant DMS-0650927,  DMS-1161456,   a grant from John Templeton Foundation. }
\address{Department of Mathematics \\ East Carolina University \\ Greenville, NC 27858}%
\email{shlapentokha@ecu.edu }
\urladdr{www.personal.ecu.edu/shlapentokha} 

\date{\today}

\maketitle

\bigskip

\bigskip

\begin{abstract}
 We consider the problem of defining polynomials over function fields of positive characteristic.  Among other results, we show that the following assertions are true.\\
 1.  Let $\G_p$ be an  algebraic extension of  a field of $p$ elements and assume $\G_p$ is not algebraically closed. Let $t$ be transcendental over $\G_p$, and let $K$ be a finite extension of $\G_p(t)$.  In this case $\G_p[t]$ has a definition (with parameters) over $K$ of the form $\forall \exists \ldots \exists P$  with only one variable in the range of the universal quantifier and $P$ being a polynomial over $K$.\\
 2.  For any $q$, for all $p \not=q$  and all function fields $K$ as above with $\G_p$ having an extension of degree $q$ and a primitive $q$-th root of unity, there is a uniform in $p$ and $K$ definition (with parameters) of $\G_p[t]$, of the form $\exists \ldots \exists \forall \forall \exists \ldots \exists P$ with only two variables in the range of universal quantifiers and $P$  being a finite collection of disjunction and conjunction of polynomial  equations over $\Z/p$.  Further, for any finite collection $\calS_K$ of primes of $K$ of fixed size $m$, there is a uniform in $K$ and $p$ definition of the ring of $\calS_K$-integers of the form $\forall\forall\exists \ldots \exists P$ with the range of universal quantifiers and $P$ as above.\\
 3.  Let $M$ be a function field of positive characteristic in one variable $t$ over an arbitrary constant field $H,$  and let $\G_p$ be the algebraic closure of a finite field in $H$.  Assume $\G_p$ is not algebraically closed.  In this case $\G_p[t]$ is first-order definable over $M$.

\end{abstract}%

\section{Introduction}
The immediate inspiration for this paper came from a classical result of R. Rumely concerning global function fields (see \cite{Rum}) and from more recent  results of B. Poonen (see \cite{PO4})  and J. Koenigsmann (see \cite{Koenig2}) for number fields.  However, the origins of the questions discussed in this paper can be traced back  to work of J. Robinson and to questions surrounding  attempts to prove that the analog of Hilbert's Tenth Problem is true for $\Q$.  

In 1949 J. Robinson gave a first-order definition of integers over $\Q$, thus proving undecidability of the first order theory of $\Q$ (see \cite{Rob1}).   Ten years later she proved that $\Z$ is definable over the ring of integers of any number field (using only {\it one} universal quantifier) and the ring of integers is definable over the number field (see \cite{Rob2}).  Thus the number fields also have an undecidable first-order theory.  J. Robinson's first definitions of integers over number fields were field-dependent and used many universal quantifiers (when the definitions are rewritten in the form $E_1 \ldots E_n P$ with $E_i$ being a universal or existential quantifier and $P$ being a system of polynomial equations).  In \cite{Rob3}, J. Robinson produced an uniform definition of $\Z$ over rings of integers of number fields. In  a 1980 paper R. Rumely updated J. Robinson's definition of integers over number fields making it uniform across number fields (see \cite{Rum}).    

The number of  universal quantifiers used in the definitions of integers became an issue in connection  to attempts to solve Hilbert's Tenth Problem over $\Q$.    Hilbert's Tenth Problem (HTP in the future) in its original form was a question posed by Hilbert at the beginning of the XX century concerning an algorithm to determine the solvability of polynomial equations over $\Z$.   In 1969, Yu. Matiasevich, building on work of M. Davis, H. Putnam and J. Robinson, showed that such an algorithm does not exist (see \cite{Da2}).  In fact, more was shown: it was proved that every r.e. subset of integers had a definition of the form $\exists \ldots \exists P$, where $P$ is again a system of polynomial equations.

  Hilbert's question of course makes sense for any recursive ring, and in particular for $\Q$.  The HTP analog for $\Q$ is currently unresolved.  One way to show that the problem is undecidable for $\Q$ is to give an existential definition of $\Z$ over $\Q$. 
Thus there is an interest in reducing (to zero if possible) the number of universal quantifiers in a definition of $\Z$ over $\Q$. B. Poonen reduced this number to two in his definition of algebraic integers over number fields while keeping the definitions uniform.  J. Koenigsmann reduced the number of quantifiers to one in his definition of $\Z$ over $\Q$.  This could very well be the optimal result as some conjectures imply that $\Z$ does not have an existential definition over $\Q$.  (See \cite{M1}, \cite{M2}, \cite{M3}, \cite{M4}, \cite{CTSDS} for a discussion of conjectures of B. Mazur on topology of rational points and \cite{Koenig2}  for a discussion of a consequence of the strong Bombieri-Lang conjecture.)

  Over function fields of positive characteristic we understand the analog of  HTP much better than over number fields.  In particular we know that the problem is undecidable over any function field of positive characteristic as long as it does not contain the algebraic closure of a finite field (see \cite{ES2}) and over function fields over algebraically closed fields of constants of transcendence degree at least 2 over a finite field (see \cite{Eis2012}).  
  
  It is natural to consider the analog of existentially defining $\Z$ over $\Q$ (or the ring of integers over a number field)  in this setting.  This analog is a construction of an existential definition of a ring of $\calS$-integers over a function field.  One could argue that the chances that such a definition exists are much higher over function  fields of positive characteristic than over number fields.
  
    First of all, one implication of a $p$-adic conjecture of B. Mazur implying the lack of an existential definition of $\Z$ over $\Q$, does not hold over function fields of positive characteristic.   More specifically, the $p$-adic conjecture of B. Mazur implies 
that there are no infinite $p$-adically discrete Diophantine subsets of $\Q$.  Of course we now know that over any function field of positive characteristic such infinite $p$-adically discrete subsets exist.  These sets are of the form $\{t^{p^s}\}$, where $t$ is some non-constant element of the field and $s \in \Z_{\geq 0}$ (see \cite{ES2}).  

Further, the author of this paper has shown that for any $\varepsilon >0$ in any global function field there is a big ring where the Dirichlet density of inverted primes is bigger than $1-\varepsilon$ and any ring of $\calS$-integers contained in the big ring has an existential definition (see \cite{Sh25}).   This existence of the function field big ring result  is in contrast to the fact that we know of no big subring of $\Q$, where we can define integers existentially.  Such subrings however, have been constructed by the author in some number fields (see \cite{Sh36}) and Poonen has constructed a {\it model} of $\Z$ in a family of big subrings of $\Q$ in \cite{Po2}.  (See \cite{PS}, \cite{EE} and \cite{EES} for other results constructing models of $\Z$ in big rings.)

   The desire for an existential definition of $\calS$-integers over a function  field of positive characteristic is motivated not just by the number field analog but also by a result of J. Demeyer (see \cite{Demeyer}) showing that over polynomial rings over finite fields of constants all r.e. sets are Diophantine.   Since polynomials are existentially definable in all rings of $\calS$-integers of global function fields (see \cite{Sh14}), an existential definition of a ring of $\calS$-integers over the field   would produce an existential definition of a polynomial ring and thus show that all r.e. subsets of the field were Diophantine.     Unfortunately we do not succeed in producing an existential definition, only a strong analog of J. Koenigsmann's result.

  In considering the simplicity of a first-order definition, besides the number of different quantifiers, one should also consider the number of parameters. R. Rumely's definition of polynomial rings over global fields uses just one (absolutely necessary) parameter.  In our definitions, driven by the agenda described above, our goal was always minimizing the number of universal quantifiers.  We were nevertheless also interested in uniform definitions and the smallest possible numbers of parameters.  However, it was not always possible to achieve all the goals in one definition.  So we generally sacrificed minimization of the number of parameters and produced uniform and non-uniform versions when a non-uniform version resulted in fewer universal quantifiers.   Before we state our main results we remark on using a single polynomial equation versus a finite collection of disjunctions and conjunctions of polynomial equations.
  \begin{remark}
 Over any not algebraically closed field a finite collection of disjunctions and conjunctions of polynomial equations can always be converted to an equivalent polynomial equation.  However, this conversion, more specifically the conversion of conjunctions or systems of polynomial equations requires a choice of a polynomial without roots in the field.  Thus across fields with different collections of polynomials without roots this conversion is not uniform.  Hence, when we are concerned about uniformity, we opt for leaving the conjunctions and disjunctions in their original multi-equation form.
  \end{remark}
  Below we state our main theorems listed in the order of generality with respect to the collection of fields covered by the results.  For the first result we sacrificed  uniformity, minimization of the number of universal quantifiers and parameters to be able to tackle an arbitrary constant field.
  
  \begin{theorem*}
  Let $M$ be a function field of positive characteristic in one variable $t$ over an arbitrary constant field $H,$  and let $\G_p$ be the algebraic closure of a finite field in $H$.  Assume $\G_p$ is not algebraically closed.  In this case $\G_p[t]$ is first-order definable over $M$.  (Theorem \ref{thm:general}.)
  \end{theorem*}
  The second theorem concerns the definition which is the analog of Koenigsmann's result.
 \begin{theorem*}
 Let $\G_p$ be an  algebraic extension of  a field of $p$ elements and assume $\G_p$ is not algebraically closed. Let $t$ be transcendental over $\G_p$, and let $K$ be a finite extension of $\G_p(t)$.  In this case $\G_p[t]$ has a (non-uniform) definition (with parameters) over $K$ of the form $\forall \exists \ldots \exists P$,  with only one variable in the range of the universal quantifier and $P$ being a polynomial over $K$. (See Theorem \ref{thm:best}.)
 \end{theorem*}  

Given an ability to define any polynomial ring within a function field $K$, we can of course define any ring of $\calS_K$-integers while not adding any universal quantifiers. All we need to do is to select a $t$ having poles at all elements of $\calS_K$ but no other poles  and define the integral closure of $\G_p[t]$ in $K$,  while reusing the one variable in the scope of the universal quantifier each time we need a variable taking values in the polynomial ring.
\begin{corollary*}
Let $K$ and $\G_p$ be as above.  In this case any ring of $\calS_K$-integers  has a definition of the form $\exists \ldots \exists\forall \exists \ldots \exists P$, with only one variable in the range of the universal quantifier and $P$ being a polynomial equation over $K$.
\end{corollary*}
 The next two results are the uniform versions of definitions above.

  \begin{theorem*}  Let $\G_p$ be an  algebraic extension of  a field of $p$ elements and assume $\G_p$ has an extension of degree $q \not =p$ for some prime $q$ and also a primitive $q$-th  root of unity. Let $K$ be a function field over  $\G_p$.    In this case for any finite collection $\calS_K$ of primes of $K$ of fixed size $m$,  there is a uniform in $p$ and $K$ definition with $m+1$ parameters  of the ring of $\calS_K$-integers of $K$ of the form $\forall\forall\exists \ldots \exists P$ with only two variables the range of universal quantifiers and $P$ being a finite collection of disjunctions and conjunctions of polynomial  equations over $\Z/p$.    (See Theorem \ref{S-integers}.) 

  \end{theorem*}
  
\begin{theorem*}
 Let $\G_p$ be an  algebraic extension of  a field of $p$ elements and assume $\G_p$ has an extension of degree $q \not =p$ for some prime $q$ and also a primitive $q$-th  root of unity. Let $t$ be transcendental over $\G_p$, and let $K$ be a finite extension of $\G_p(t)$.   In this case there is a uniform in $p$ and $K$ definition (with four parameters) of $\G_p[t]$, of the form $\exists \ldots \exists \forall \forall \exists \ldots \exists P$ with only two variables in the range of universal quantifiers and $P$  being a finite collection of disjunction and conjunction of polynomial  equations over $\Z/p$.  

  \end{theorem*}  
    If we restrict ourselves to global fields, then for all global fields of characteristic not equal to 2, we can use $q=2$ and for $p=2$, assuming the constant field contains $\F_4$, we could use $q=3$.  
\begin{corollary*}
Let $K$ be a global field with a constant field of odd size or divisible by 4.   In this case a ring of $\calS_K$-integers,  where we  fix the size of $\calS_K$, and any polynomial ring contained in $K$ have uniform in $K$ and $p$ definitions (with parameters) as described above. 
\end{corollary*}
If we want to cover all global fields, then we have to add a primitive third of unity to characteristic 2 fields and end up with an extra universal quantifier.  

\section{How Do Our Results Compare to Rumley's and an Overview of the Proof.}
\subsection{New vs. Old}    
As we have mentioned in the introduction, one of the goals of our construction was to produce definitions of polynomial rings of the  of the form 
\[
E_1 x_1 \ldots E_n x_nP(t, \bar w, \bar x),
\]
 where each $E_i$ is either a universal or an existential quantifier,  $P$ is a polynomial equation or a system of polynomial equations, and $\bar w$ is a vector of parameters, while minimizing the number of universal quantifiers.  It is natural to ask what happens to the definitions constructed by R. Rumely if we rewrite them in this (prenex normal) form.  We carried out some of the rewriting in a direct manner in the appendix and obtained a lower bound on the order of 16 variables within the range of a universal quantifier for  a definition of a ring of integral functions or a polynomial ring.  So in this respect, our definitions do much better.  Our uniform definition for polynomial rings and the rings of $\calS$-integers requires two universal quantifiers, and the non-uniform version requires just one.  

The other aspect of improvement in our results is the scope of our formulas.  R. Rumely considered only global function fields, i.e. function fields in one variable with finite constant fields.  For our uniform and one-universal-quantifier definitions we consider function fields of transcendence degree one over any field algebraic over a finite field and not algebraically closed.  We further consider arbitrary function fields not containing the algebraic closure of a finite field, though in the most general case we give up on counting universal quantifiers and uniformity.

R. Rumely's formula remain superior to ours in one important way: the number of parameters.  He uses only one.  We need at least four.

\subsection{The Main Tools}
We use two main tools: $p$-th power equations and norm equations.  Let $K$ be any field of characteristic $p>0$ and let $p(K)=\{(x, x^{p^m})| x \in K, m \in \Z_{\geq 0}\}$.  By $p$-th power equations we mean a Diophantine defintion  of $P(K)$ (or an existential defintion of $P(K)$ in the language of rings) and we have it over any function field of positive characteristic (see \cite{ES2}).  Unfortunately, this definition is very much field dependent (so non-uniform).  To get to a uniform definition  of $P(K)$ we had to make a detour to a ring of $\calS$-integers with only one prime allowed in the denominator. Over these rings the $p$-th power equations are uniform, but the detour required an extra universal quantifier.  We return to the uniformity issue later, and now explain the role of $P(K)$ in our definition of polynomials.  
\subsubsection{Using $p$-th power extensions to get rid of primes in the denominator} 

We start with $K=\F_p(t)$, where $\F_p$ is  a finite field of $p$ elements and $q$ is a rational prime  which may be equal to $p$.  Let $y \in \F_p(t)$ and consider the following equation as $m \rightarrow \infty$:
\[
w_m=\frac{y^{qp^m}-y^q}{t^{p^m}-t}.
\]
Now,
\be
\item \label{it:1}If $y$ is a polynomial in $t$, then for all $m \in \Z_{>0}$ we have that $w_m$ is a polynomial in $t$. (See Lemma \ref{will work}.)
\item \label{it:2} For any prime $\Aa$ of $\F_p(t)$ such that $\Aa$ is not the pole of $t$, we have that for some $m_0 \in \Z_{>0}$ for all positive integer $m$ divisible by $m_0$ it is the case that $\ord_{\Aa}(t^{p^m}-t)=1$.
\item \label{it:3} If $\Aa$ is a pole of $y$, then $\ord_{\Aa}(y^{qp^m}-y) \equiv 0 \mod q$.
\ee    
Let $ \Qq$ be  the pole of $t$.  Now consider the following set
\[
\left \{y \in \F_p(t)| \forall m \in \Z_{>0} \forall \Aa \not =\Qq: \ord_{\Aa}\frac{y^{qp^m}-y^q}{t^{qp^m}-t} \geq 0 \lor \ord_{\Aa}\frac{y^{qp^m}-y}{t^{p^m}-t}\equiv 0 \mod q\right \} 
\]
We claim that the elements in this set are precisely the polynomials in $t$.  If $\Aa\not =\Qq$ and $\ord_{\Aa}y <0$, then for some $m$ we have that $\ord_{\Aa}(t^{p^m}-t)=1$ and
\[
\ord_{\Aa}\frac{y^{qp^m}-y^q}{t^{p^m}-t}=qp^m\ord_{\Aa}y-1 \not \equiv 0 \mod q.
\]
As we have pointed out above, the polynomials are in this set.  So the set must coincide with the polynomial ring.  Now we note the following.
\be
\item To accommodate bigger constant fields we can allow for all $m$ to be divisible by a constant $m_0$, if the constant field is of the size $p^{m_0}$.  If the constant field is infinite then we can rephrase the requirement for $m$ to ask for the existence of a constant $m_0$ depending on $y$ so that for all $m$ divisible by $m_0$ the order conditions are satisifed.
\item Instead of quantifiying over $m$ we can state the following:
\[
\forall z \in \F_p(t): (\ord_{\Qq}z \geq 0) \lor (\exists t^{p^m}: \ord_{\Qq}t^{p^m} > \ord_{\Qq}z >\ord_{\Qq}t^{p^{m+1}}) \lor
\]
\[
 (\exists t^{p^m}: \ord_{\Qq}t^{p^m} = \ord_{\Qq}z  \land \forall \Aa \not = \Qq: \ord_{\Aa}\frac{y^{qp^m}-y^q}{t^{p^m}-t} \geq 0 \lor \ord_{\Aa}\frac{y^{qp^m}-y^q}{t^{p^m}-t}\equiv 0 \mod q)
 \]
 \item Raising $t$ and $y^q$ to the same $p^m$-th power while operating over a field is explained in Lemma \ref{prop:pthpower2}.
 \item The condition 
\begin{equation}
\label{eq:order}
\forall \Aa \not =\Qq: \ord_{\Aa}\frac{y^{qp^m}-y^q}{t^{p^m}-t} \geq 0 \lor \ord_{\Aa}\frac{y^{qp^m}-y^q}{t^{p^m}-t}\equiv 0 \mod q
\end{equation}
 is written down  using a norm equation as will be explained below.  This statement also requires a universal quantifier but the Strong Approximation Theorem allows us to reuse the same $z$.
 \item All the requirements concerning the order at a single prime can be stated as existential statements.
 \ee
 \subsubsection{Using $p$-th power equations to pick out rational functions in algebraic extensions}
 To simplify the presentation let $K/\F_p(t)$ be a cyclic extension and let $\Aa \not | \Qq$ be a prime of $K$.  (We now think of $\Qq$ as a not necessarily prime divisor.)  Now let $y \in K$ and let's see what \eqref{eq:order} for some positive integer $m$ rewritten over $K$ tells us about $y$.  
\begin{equation}
\label{eq:orderabove}
\mbox{For every prime } \aaa \mbox{ of } K \mbox{ not dividing } \Qq  \mbox{ and not ramified over } \F_p(t): 
\end{equation}
\[
\ord_{\aaa}\frac{y^{qp^m}-y^q}{t^{p^m}-t} \geq 0 \lor \ord_{\aaa}\frac{y^{p^m}-y}{t^{p^m}-t}\equiv 0 \mod q
\]
As before, we conclude $\ord_{\aaa}y \geq 0$, but we actually get more information about $y$.  We also conclude that $\ord_{\aaa}(y^{qp^m} -y^q) >0$ for all $\aaa$ dividing $t^{p^m}-t$.  To take full full advantage of this information, we need  to expand the field $K$ by adjoining $\F_{p^m}$ -- finite field of $p^m$ elements.  Now the extension $K\F_{p^m}/\F_{p^m}(t)$  is still cyclic.  By an effective version of Chebotarev Density Theorem we know that for all sufficiently large $m$ we have degree one primes which do not split in this extension and the number of such primes increases to infinity as $m \rightarrow \infty$.  (We have explicit formulas giving the lower bound on the number of degree one primes not splitting in such an extension depending on $m$ and $[K:F]$.)  So let $\aaa$ be a degree one prime lying above a rational prime $\Aa$ in $\F_{p^m}(t)$ and not splitting in the extension $\F_{p^m}K/\F_{p^m}(t)$.  Thus $\aaa$ must correspond to a linear polynomial $t-a, a \in \F_{p^m}$ and $\ord_{\aaa}(y^{qp^m}-y^q) >0$.  Factoring $y^{qp^m}-y^q$, we conclude that for some $b \in \F_{p^m}$ it is the case that $\ord_{\aaa}(y^q-b) >0$.  Now if we have sufficiently many such primes $\aaa$ relative to the height of $y^q$ (or in other words $m$ is sufficiently large relative to the height of $y$), by the Weak Vertical Method we can conclude that $y^q \in \F_{p^m}(t)$.  (The details of the application of the Weak Vertical Method and Chebotarev Density Theorem with all the relevant references are in Section \ref{subsec:weak}.)

To get back to the constant field of the required size (when this size is less than $p^m$) we can still use $p$-th power equations.  (See Lemma \ref{rational}.)  A few more equations may be required to conclude that $y$ (as opposed to $y^q$) is in the rational field.  (See, for example, Equations \eqref{eq:qth1}, \eqref{eq:x2} in the one-universal-quantifier definition of polynomials.)  Finally we note that a complete one-universal-quantifier definition of polynomials over a function field $K$ of transcendence degree one over a finite field and not containing the algebraic closure of a finite field is in Theorem \ref{thm:best}.
 \subsubsection{Using $p$-th power equations to pick out algebraic functions in a transcendental extension}
When the extension is transcendental we have a ``transcendental'' version of the Weak Vertical Method in Lemma \ref{le:algebraic}.  To simplify the discussion we again return  to a rational function field.  Let $H(t)$ be a rational field containing $\F_p(t)$, and let $h \in H(t)$ be such that  for infinitely many $a$ algebraic over $\F_p$ we have that $h(a)$ is algebraic over $\F_p$.  In this case, by Lemma \ref{le:algebraic}, we conclude that $h \in \F_p(t)$.  So, as above, it is enough to have for infinitely many algebraic $\Aa$ (i.e. $\Aa$ corresponding to prime polynomials with coefficients algebraic over $\F_p$),
\begin{equation}
\label{eq:order2}
 \ord_{\Aa}\frac{y^{qp^m}-y^q}{t^{p^m}-t} \geq 0 \lor \ord_{\Aa}\frac{y^{qp^m}-y^q}{t^{p^m}-t}\equiv 0 \mod q
\end{equation}
Using norm equations one can require that the statement above holds for infinitely many algebraic primes.  A first-order definition of algebraic polynomial rings (i.e. of transcendence degree one over a finite field) over arbitrary function fields not containing the algebraic closure of a finite field can be found in  Proposition \ref{prop:definable}
\subsubsection{Using Norm Equations}
Our use of norm equations is based on the Hasse Norm Principle as was the case for R. Rumely.  However, we do employ  a unique,  to our knowledge,  variation of the norm method.  More specifically, as explained below, we do not fix the top or the bottom field in the norm equation, but allow these fields to vary depending on the elements involved.  As long as the degree of all extensions involved is bounded, such a ``floating'' norm equation is still (effectively) translatable into a system of polynomial equations over the given field.  (See the section on coordinate polynomials in the appendix of  \cite{Sh34} for a general and formal discussion of  rewriting techniques and the proof of Theorem \ref{S-integers}  for a more informal description of this translation).   

To explain our use of norm equations, we again, as above, make some simplifying assumptions.  More specifically let 
\begin{itemize}
\item $q$ be a rational prime number different from the characteristic of the field,
\item $K$ be a global function field containing a primitive $q$-th root of unity,
\item $c \in K$ be such that it is not a $q$-th power but has order at every prime divisible by $q$,
\item $x \in K$ be such that  all zeros of $x$ are of order divisible by $q$.
\end{itemize}
  Now consider the norm equation
\begin{equation}
\label{norm:prelim}
{\mathbf N}_{K(\sqrt[q]{c})/K}(y)=x.
\end{equation}
Using the Hasse Norm Principle, under our assumptions, this norm equation has solution if only if every pole  $\pp_K$ of $x$ either splits in the extension $K(\sqrt[q]{c})/K$ or has order divisible by $q$.  Indeed, consider the following.  Our conditions on $c$ insure that the extension is unramified.  So if at some prime $\pp_K$ the local degree is not equal to one, i.e. if the prime does not  split, the norm equation has solution if and only if the order at this prime is divisible by $q$.  If the order is divisible by $q$, then locally $x=\varepsilon\pi^{mq}$, where $\varepsilon$ is a unit, $\pi$ is a local uniformizing parameter and $m \in \Z$.  By the Local Class Field Theory, in an unramified extension of global fields, every unit is a norm, and in the extension of degree $q$, every $q$-th power below is a norm also.  

For an arbitrary $c$ and $x$ in $K$, we will not necessarily have  all  zeros of $x$ and all zeros and poles of $c$ of orders divisible by $q$.   For this reason, given $x, c \in K$ we would like to consider our norm equation in a finite extension $L$ of $K$ and this extension $L$ depends on $x, c$ and $q$.  We would like to choose $L$ so that all  primes occurring as zeros of $x$ or as  zeros or poles of $c$ are ramified with ramification degree divisible by $q$ and all the primes that are poles of $x$ split completely in $L$, so that in $L$ we still have that $c$  is not a $q$-th power modulo any factor of $\pp_L$.   If $x$ has poles only at ``allowed '' primes, we can always select $c$ so that all the requirements for $L$ can be satisfied at the same time.  But if $x$ has ``unallowed'' poles we may fail to ramify some primes which are zeros or poles of $c$ and as a result the norm equation will have no solution, but that is precisely the situation we want to be in. 

 This way, as we run through all $c \in K$ with $c-1 \equiv  0 $ modulo all ``allowed'' poles so that they split when we take the $q$-th root of $c$, we ``catch'' all the ``unallowed'' primes  that occur as poles of $x$ of order not divisible by $q$.  The construction of the field $L$ and the argument concerning the properties of the primes in question in this field are in Proposition \ref{prop:norm2}.  Note that if we are prepared to allow $x$ to have ``unallowed'' poles as long as the order is divisible by $q$, we need essentially one universal quantifier and the rest of the ``clean up'' is done via the $p$-th power equations as explained above.  We can also clean up all the ``undesirable'' poles by using only a norm equation and uniformly in various classes of fields, but this will require another quantifier.  To follow that option we change the right side of \eqref{norm:prelim} to  
\begin{equation}
\label{norm:prelim2}
{\mathbf N}_{K(\sqrt[q]{c})/K}(y)=bx^q+b^q
\end{equation}
Here we look at the case of a prime $\pp_K$ occurring as a pole of $x$ and such that $\ord_{\pp_K}b=-1$.  Under these conditions we have $\ord_{\pp_K}bx^q+b^q \not \equiv 0 \mod q$, and if $c$ is not a $q$-th power modulo $\pp_K$, the norm equation \eqref{norm:prelim2} will not have a solution.  This equation allows us to catch all ``unallowed'' poles but we need to quantify universally over $c$ and $b$.

As as is usually the case we need to deal with the situations where the all the algebraic extensions of the constant field are of degrees divisible by the characteristic separately.  Instead of using $p$-th roots we use extensions generated by polynomials of the form $T^p-aT+1$.  While the argument is technically different in this case, ideologically it is the same.  It is carried out in Section \ref{degp}.

\section{Some Algebraic  Number Theory.}
\label{Algebraic}
\subsection{Primes and Valuations}
In this section we discuss briefly valuations or primes of function fields and establish notation to be used below.  

By a {\it function field},  we mean a function field in one variable. By a {\it global function field} we mean a function field in one variable over a finite field of constants.
By a {\it valuation} $v$ of a function field $K$ we will mean a {\it discrete valuation} of $K$ which is trivial on the constant field.  Given a valuation $v$, we can consider its {\it valuation ring} $R_v=\{x \in K| v(x) \geq 0\}$ and the unique maximal ideal of this ring $\pp(v)=\{x \in K| v(x)>0\}$.  We will often identify $v$ and $\pp(v)$ and use the terms ``valuations of $K$'' and ``primes of $K$'' interchangeably.  

Given an element $x \in K$, we will say that $x$ has a zero at a prime $\pp(v)$ (or a valuation $v$) if $\ord_{\pp(v)}x >0$, where we define the order in the usual way:  that is, if $x \in R_v, x \not=0$, then $\ord_{\pp(v)}x$ is the largest integer $n$ such that $x \in \pp(v)^n$ but $x \not \in  \pp(v)^{n+1}$, and $\ord_{\pp(v)}0=\infty$. If $x \not \in R_v$, then by definition of a valuation, $\frac{1}{x} \in R_v$ and we define $\ord_{\pp(v)}x = -\ord_{\pp(v)}\frac{1}{x}$.  If $\ord_{\pp(v)}x <0$, then we say that $x$ has a pole at $\pp(v)$.  For more details concerning the valuations and primes of function fields see for example \cite{C}.\\

\begin{notationassumptions}
\label{not:1}
The following notation are used throughout the rest of the paper.
\begin{itemize}
\item Let $p \not = q$ be prime numbers.  We reserve $p$ as a notation for the characteristic of the fields in the paper.
\item Let $\F_{p^s}$ be a finite field of $p^s$ elements (of characteristic $p$) and let $\tilde{\F}_p$ denote the algebraic closure of $\F_p$.
\item Let $\G_p$ be an algebraic extension of $\F_p$.

\item Let $\xi_q$ be a $q$-th primitive root of unity.
\item Let $t$ be transcendental over $\F_p$.
\item Let $K, F, G, H, M$ denote finite algebraic extensions of $ \G_p(t)$.  Assume $\G_p$ is algebraically closed in $K,F,G,H$.  We will refer to these fields as {\it function fields} omitting ``in one variable of positive characteristic over the constant field $ \G_p$''.
\item Let $\tilde K=\tilde G, \ldots$ denote the algebraic closure of $\F_p(t)$.
\item For a function field $G$, let $\pp_G, \qq_G, \ttt_G, \aaa_G, \ldots$ be distinct primes of $G$.
\item If $H$ is any finite extension of a function field $G$, then let $\pp_H, \qq_H, \ttt_H, \aaa_H, \ldots$  denote primes above $\pp_G, \qq_G, \ttt_G, \aaa_G, \ldots$ respectively.
\item If $\calS_K$ is a set of primes of a function field $K$, then we let $O_{K,\calS_K}$ denote a subring of $K$ containing all the elements of $K$ without any poles at primes outside $\calS_K$.  If $\calS_K$ is finite, then we call the ring $O_{K,\calS_K}$ {\it the ring of $\calS_K$-integers}.

\item Given an element $x \in K$, let $\sqrt[q]{x}$ denote an element of  $\tilde K$ whose $q$-th power is $x$.  If $K$ contains such an element, then assume $\sqrt[q]{x} \in K$.

\item For any prime $\pp_K$, let  $K_{\pp_K}$ be the completion of $K$ under the $\pp_K$-adic topology.
\end{itemize}
\end{notationassumptions}

Next we state Hensel's lemma and its corollary which play an important role in the use of the Hasse Norm Principle.
\begin{lemma}
If $K$ is a function field, $f(X) \in K_{\pp_K}[X]$ has coefficients integral at $\pp_K$ and for some $\alpha \in K_{\pp_K}$ integral at $\pp_K$ we have that $\ord_{\pp_K}f(\alpha) > 2\ord_{\pp_K}f'(\alpha)$, then $f(X)$ has a root in $K_{\pp_K}$.  (See \cite{L}[Proposition 2, Section 2, Chapter II].)
\end{lemma}
\subsection{Using Extensions of Degree $q \not=p$ to Define Integrality}
In this section we make the following assumption.
\begin{assumption}
\label{ass:pnotq}
Assume $ \G_p$ has an extension of degree $q$ and contains $\xi_q$.
\end{assumption}
\begin{lemma}
\label{le:first}
 If $b \in M$ and $b$ is not a $q$-th power in $M$, then the following statements are true.
\begin{enumerate}
\item \label{firstit:0} $[M(\sqrt[q]{b}):M]=q$.
\item \label{firstit:1} If $\ord_{\pp_M}b\equiv 0\mod q$, then $\pp_M$ does not ramify in the extension $M(\sqrt[q]{b})/M$.
\item \label{firstit:2} If $b$ is not a $q$-th power  mod $\pp_M$, and $\ord_{\pp_M}b=0$, then $\pp_M$ does not split (i.e. has only one prime above it) in the extension $M(\sqrt[q]{b})/M$.
\item \label{firstit:3} If $b$ is a $q$-th power  mod $\pp_M$ and $\ord_{\pp_M} b =0$, then $\pp_M$ splits in the extension $M(\sqrt[q]{b})/M$.
\item \label{firstit:4} If $\ord_{\pp_M}b \not \equiv 0 \mod q$, then $\pp_M$ ramifies completely in the extension $M(\sqrt[q]{b})/M$.
\end{enumerate}
\end{lemma}
The second lemma deals with norms and primes in cyclic extensions of degree $q$.
\begin{lemma}
\label{le:cycextensions}
Let $M_1/M_2$ be a cyclic extension of degree $q$.  If $\pp_{M_2}$ is not ramified in the extension, then either it splits completely (in other words into $q$ distinct factors) or it does not split at all.  Further if $w ={\mathbf N}_{M_1/M_2}(z)$ for some $z \in M_1$ and $\pp_{M_2}$ does not split in the extension,  then $\ord_{\pp_{M_2}}w \equiv 0 \mod q$.
\end{lemma}
We now add to our notation list.
\begin{notation}
\begin{itemize}
\item Let $\calS_K=\{\pp_{1,K}\ldots,\pp_{l, K}\}$ be a finite set of primes of a function field $K$. 
\item Let $\Phi(K,\calS_K)$,  denote the set of all elements $c$ of $K$ such that  $\ord_{\pp_{i,K}}(c-1)>0$ for all $i=1,\ldots, l$.  If $\calS_K =\emptyset$, then set $\Phi(K,\calS_K)=K$.  
\end{itemize}
  \end{notation}
The next two propositions introduce extensions $L_1, L_2$, and $L_3$ and explain their purpose: ramifying all  zeros of $x$ and $bx^q+b^q$ and all zeros and poles of $c$ to avoid ramifying primes in the cyclic extension where we are going solve the norm equation, and to make sure that zeros of $x$ do not have any influence on whether the norm equation has solutions.
\begin{proposition}
\label{prop:badprime}
Let $x,b, c  \in K, x \not=0$, $c\not=0$,  $bx^q+b^q\not=0$, 
\[
L_1=K(\sqrt[q]{1+ x^{-1}}),
\]
\[
L_2=L_1(\sqrt[q]{1+ (bx^q+b^q)^{-1}})
\]
\[
L_3=L_2(\sqrt[q]{1+(c+c^{-1})x^{-1}}).
\]
If for some prime $\pp_K$ the following assumptions are true:
\be
\item \label{ass:2} $c$ is not a $q$-th power modulo $\pp_K$ (note that this assumption includes the assumption that  $\ord_{\pp_K}c = 0$),
\item \label{ass:3}$\ord_{\pp_K}x <0$,
\item \label{ass:4}$\ord_{\pp_K}b \not \equiv 0 \mod q, $
\item \label{ass:last} $q\ord_{\pp_K}x < (q-1)\ord_{\pp_K}b$,
\ee
then for every prime factor $\pp_{L_3}$ of $\pp_K$ in $L_3$ we have that
\be
\item $\ord_{\pp_{L_3}}x <0$,
\item $c$ is not a $q$-th power modulo $\pp_{L_3}$ and thus not a $q$-th power in $L_3$,
 and
\item $\ord_{\pp_{L_3}}(bx^q + b^q)\not \equiv 0 \mod q$.
\ee
\end{proposition}
\begin{proof}
If $\pp_K$ is a $K$-prime as described by the statement of the proposition, then by Assumption \ref{ass:3} we have that
 \[
 \ord_{\pp_{K}}(x^{-1}) >0
 \]
 and therefore by Lemma \ref{le:first}, Part \ref{firstit:3}  we have that $\pp_{K}$ splits completely into distinct factors in the extension $L_1/K$. (We remind the reader that $L_1= K(\sqrt[q]{1+ x^{-1}})$.)  Thus, in $L_1$ we have that $\ord_{\pp_{L_1}}x <0$, $\ord_{\pp_{L_1}}b \not \equiv 0 \mod q$, and $c$ is not a $q$-th power modulo $\pp_{L_1}$.  We now note that by Assumption \ref{ass:last} we have that
\[
q\ord_{\pp_K} x+\ord_{\pp_K}b < q\ord_{\pp_K}b,
\]
and therefore
 \[
 \ord_{\pp_{L_1}}(bx^q+b^q) =\ord_{\pp_{L_1}}b +q\ord_{{L_1}}x <0.
 \]
  Further, by Assumption \ref{ass:4} we have that $\ord_{\pp_{L_1}}(bx^q+b^q) \not \equiv 0 \mod q$.  Applying Lemma  \ref{le:first} Part \ref{firstit:3} again, this time over the field $L_2 =L_1(\sqrt[q]{1+ (bx^q+b^q)^{-1}})$, we see that in the extension $L_2/L_1$, the $L_1$- prime $\pp_{L_1}$ splits completely into distinct factors and thus $c$ is not a $q$-th power modulo any $\pp_{L_2}$, while  $$\ord_{\pp_{L_2}}(bx^q+b^q) \not \equiv 0 \mod q$$ and  $$\ord_{\pp_{L_2}}(bx^q+b^q) <0.$$   

Since, by assumption, $\ord_{\pp_{K}}c =0$ and therefore $\ord_{\pp_{L_2}}c =0$, by Lemma \ref{le:first}, Part \ref{firstit:3} one more time, $\pp_{L_2}$ will split completely into distinct factors in the extension $L_3/L_2,$ and,  as before this would imply that $c$ is not a $q$-th power in $L_3$ or modulo any of $\pp_{L_3}$.  Here we remind the reader that  $$L_3 =L_2(\sqrt[q]{1+(c+c^{-1})x^{-1}}).$$  Finally, we also have \[ \ord_{\pp_{L_3}}(bx^q+b^q) \not \equiv 0 \mod q.\]
\end{proof}
\begin{proposition}
\label{prop:fixorder}
Let $x, b, c, L_1,L_2, L_3$ be as in Lemma \ref{prop:badprime}. In this case, for any $L_3$-prime $\aaa_{L_3}$ that is not a pole of $x$, the following statements hold:
\be
\item $\ord_{\aaa_{L_3}}c \equiv 0 \mod q$;
\item $\ord_{\aaa_{L_3}}(bx^q+b^q) \equiv 0 \mod q$.
\item $\ord_{\aaa_{L_3}}x \equiv 0 \mod q$
\ee
\end{proposition}
\begin{proof}
We again proceed by applying Lemma \ref{le:first} three times.  In the extension $L_1/K$, where $L_1= K(\sqrt[q]{1+ x^{-1}})$,  all the zeros of $x$ that are not of order divisible by $q$ are ramified by Lemma \ref{le:first}, Part \ref{firstit:4}, since for any $K$-prime $\aaa_{K}$ such that $\ord_{\aaa_{K}}x>0$ we have that $\ord_{\aaa_{K}}(1+x^{-1})<0$. 

  In the extension $L_2/L_1$, where  $L_2=L_1(\sqrt[q]{1+ (bx^q+b^q)^{-1}})$, as before, we ramify all the primes $\aaa_{L_1}$ such that $\ord_{\aaa_{L_1}}(bx^q+b^q)>0$ and $$\ord_{\aaa_{L_1}}(bx^q+b^q) \not \equiv 0 \mod q.$$ Further, if $\aaa_{L_1}$ is a pole of $bx^q+b^q$ but not a pole of $x$, then it is a pole of $b$ and therefore $\ord_{\aaa_{L_1}(bx^q+b^q)}  = q\ord_{\aaa_{L_1}}b$.
  
  Finally,  $(c+c^{-1})x^{-1}$ has poles at all primes occurring in the divisor of $c$ and not poles of $x$.  Since in $L_1$, and therefore in $L_2$, all zeros of $x$ are of order divisible by $q$, if $c$ has a pole or a zero of degree not divisible by $q$, and the prime in question  is not a pole of $x$, it follows that $(c+c^{-1})x^{-1}$ has a pole of degree not divisible by $q$ at this prime, forcing it to ramify  in the extension $L_2(\sqrt[q]{1+(c+c^{-1})x^{-1}})/L_2$.  Thus, $\ord_{\aaa_{L_3}}c \equiv 0 \mod q$ for any prime $\aaa_{L_3}$ not being a pole of $x$.
\end{proof}

\begin{proposition}
\label{prop:norm}
If $K$ is a function field, $x, b, c $, $L_3$ is as in Proposition \ref{prop:badprime}, $b \in K$, $\calS_K$ is finite set of primes of $K$,  $c \in  \Phi(K,\calS_K)$,   and there exists $y \in L_3(\sqrt[q]{c})$ such that
\begin{equation}
\label{eq:norm}
{\mathbf N}_{L_3(\sqrt[q]{c})/L_3}(y) = bx^q+b^q,
\end{equation}
then for any non-archimedean prime $\pp_K$ of $K$ it is the case that one of the following conditions holds:
\be
\item \label{nass:1}$c$ is a $q$-th power mod $\pp_K$, or
\item \label{nass:2}$\ord_{\pp_K} x \geq 0$, or
\item \label{nass:3}$q\ord_{\pp_K}x  \geq (q-1)\ord_{\pp_K}b$, or
\item \label{nass:4}$\ord_{\pp_K}b \equiv 0 \mod q$.
\ee
Conversely, if $x \in O_{K,\calS_K}$,  then \eqref{eq:norm} has a solution $y \in L_3(\sqrt[q]{c})$.
\end{proposition}
\begin{proof}
If for some $K$-prime $\pp_K$  we have that none of the Conditions \ref{nass:1} -- \ref{nass:4} is satisfied,  then by Proposition \ref{prop:badprime}, we have that $\ord_{\pp_{L_3}}(bx^q+b^q)\not \equiv 0 \mod q$ and $c$ is not $q$-th power modulo $\pp_{L_3}$.  Hence by by Lemma \ref{le:cycextensions} we conclude that the norm equation \eqref{eq:norm} has no solution in $ L_3(\sqrt[q]{c})$.    
 
Suppose now that $x \in O_{K,\calS_K}$.    In this case by Proposition \ref{prop:fixorder}, for  every prime $\aaa_{L_3}$, not dividing any prime in $\calS_K$, we have the following:
\begin{itemize}
\item  $\ord_{\aaa_{L_3}}(bx^q+b^q) \equiv 0 \mod q$,
and
\item  $\ord_{\aaa_{L_3}}c \equiv 0 \mod q.$
\end{itemize} 
Further, since the divisor of $c$ is a $q$-th power in $L_3$, the extension $L_3(\sqrt[q]{c})/L_3$
 is unramified at all primes by Lemma \ref{le:first}.

 By the Hasse's Norm Principle (see Theorem 32.9 of \cite{Reiner}) this norm equations has solutions globally (i.e. in $L_3$) if and only if it has a solution locally (i.e. in every completion).   (Note that while Hasse Norm Principle is stated for global function fields, it is clearly applicable here since we can find a global function field contained in $L_3$ containing the values of all the variables under consideration.)

 Observe further that locally every unit is a norm in a unramified extension (see Proposition 6, Section 2, Chapter XII of \cite{W}).   

   Next we observe that since $L_3(\sqrt[q]{c})/L_3$ is a cyclic extension of prime degree, by Lemma \ref{le:cycextensions} every unramified prime either splits completely or does not split at all.  If a prime splits completely, then the local degree is one and every element of the field below is automatically a norm locally at this prime.  So the only primes where we might have elements that are not local norms are the primes that do not split, or, in other words, the primes where the local degree is $q$.   (Note that any factor of a prime in $\calS_K$ splits completely in the extension $L_3(\sqrt[q]{c})/L_3$ by our assumptions on $c$ and Lemma \ref{le:first}.) 

So let $\pp_{L_3}$ be  a prime of local degree $q$ and not lying above a prime of $\calS_K$.  By the argument above we have that $\ord_{\pp_{L_3}}(bx^q +b^q) \equiv 0 \mod q$.  In this case, by the Weak Approximation Theorem, there exists $u \in L_3$ such that $\ord_{\pp_{L_3}}u =1$  and therefore for some integer $m$ it is the case that  $u^{qm}(bx^q+b^q)$ has order 0 at $\pp_{L_3}$ or in other words $u^{qm}(bx^q+b^q)$ is a unit at $\pp_{L_3}$.   
As any $q$-th power of an $L_3$-element, $u^{mq}$ is a norm locally since the degree of the local extension is $q$ by our assumption.  Therefore, $u^{mq}(bx^q+b^q)$ is a norm at $\pp_{L_3}$ if and only if $(bx^q+b^q)$ is a norm at $\pp_{L_3}$.  But $u^{mq}(bx^q+b^q)$ is a unit at $\pp_{L_3}$ and therefore is a norm.  Hence $bx^q+b^q$ is a norm.
\end{proof}
In the same fashion one can prove another version of Proposition \ref{prop:norm}.
\begin{proposition}
\label{prop:norm2}
If $K, L_1, x, c$ are as above,  let $L_4=L_1(\sqrt[q]{1+(c+c^{-1})x^{-1}})$.  Further assume $\calS_K$ is finite set of primes of $K$,  $c \in  \Phi(K,\calS_K)$,   and there exists $y \in L_4(\sqrt[q]{c})$ such that
\begin{equation}
\label{eq:norm2}
{\mathbf N}_{L_4(\sqrt[q]{c})/L_3}(y) = x,
\end{equation}
then for any prime $\pp_K$ of $K$ it is the case that one of the following conditions hold:
\be
\item \label{nass1:1}$c$ is a $q$-th power mod $\pp_K$, or
\item \label{nass1:2}$\ord_{\pp_K} x \geq 0$, or
\item \label{nass1:4}$\ord_{\pp_K}x \equiv 0 \mod q$.
\ee
Conversely, if $x \in O_{K,\calS_K}$,  then \eqref{eq:norm} has a solution $y \in L_4(\sqrt[q]{c})$.
\end{proposition}

\subsection{Using Extensions of Degree $p$ to Define Integrality}
\label{degp}
We now address the issue of extensions of degree $p$ and their use for norm equations defining integrality.  Consequently, we drop the assumption that $\G_p$ has an extension  of degree $q \not =p$ and $\xi_q$ and assume the following.
\begin{assumption}
$\G_p$ has an extension of degree $p$. 
\end{assumption}
\begin{lemma}
\label{le:finite}
The extension of $\G_p$ of degree $p$ is generated by an element satisfying an equation of the form $T^p-a^{p-1}T-1=0$ for some $a \in F \setminus \{0\}$.
\end{lemma}
\begin{proof}
Since any element generating an extension of degree $p$ over $\G_p$ will also be generating an extension of degree $p$ over the largest finite field contained in $\G_p$ of size $p^k$ with $ k \in \Z_{>0}$, and since any finite field has a unique extension of any degree, it is enough to show that any finite field has an extension of degree $p$ generated by a root of a polynomial of the form $P(T)=T^p-a^{p-1}T-1$, with $a$ being a non-zero element of the field.  We demonstrate this by counting the polynomials and potential solutions in a finite field.  Observe that if $\alpha$ satisfies $P(T)$, then $(\alpha + ia)$ is also a root of $P(T)$ for $i=0,\ldots, p-1$.  Further, for non-zero $a, b$ in our finite field,  $a^{p-1}=b^{p-1}$ if and only if $a=i b$, where $i=1, \ldots,p-1$ and if $Q(T)=T^p-b^{p-1}T-1$, then $P(T)$ and $Q(T)$ have a common root if and only if $a^{p-1}=b^{p-1}$.  Thus, if the field has $p^k$ elements, we have $\frac{p^k-1}{p-1}$ polynomials with pairwise non-intersecting root sets.   Note that 0 is not a root of any of these polynomials.   If all of these polynomials have roots in our finite field, then 
\[
p\frac{p^k-1}{p-1} \leq p^k-1,
\]
\[
p^{k+1}-p \leq p^{k+1}-p-p^{k}+1,
\]
 and 
\[
0\leq 1-p^{k}.
\]
  Thus, at least one polynomial $P(T)$ does not have a root in our finite field and its root generates the extension of degree $p$.
\end{proof}
\begin{lemma}
\label{lemma:allsol}
Let $a \in K\setminus \{0\}$ and let $\alpha$ be a root of the equation
\begin{equation}
\label{eq:5}
T^p - a^{p-1}T- 1 = 0.
\end{equation}
In this case either $\alpha \in K$ or $\alpha$ is of degree $p$ over
$K$.  In the second case the extension $K(\alpha)/K$ is cyclic of
degree $p$ and the only primes possibly ramified in this extension are zeros of $a$.  More precisely, if for some $K$-prime $\aaa_K$, we have that
$\ord_{\aaa_K}a \not \equiv 0$ modulo $p$ and $\ord_{\aaa_K}a>0$,  then a factor of ${\aaa_K}$ in $K(\alpha)$ will be
ramified completely. Further, if $\ttt_K$ is a pole of $a$, then $\ttt_K$ splits completely in the extension.
\end{lemma}
\begin{proof}
 As we have already observed in the proof of Lemma \ref{le:finite}, if $\alpha$ is a root of
  (\ref{eq:5}) in the algebraic closure of $K$, then $\alpha +ia, i=0,\ldots,p-1$ are also roots.    Hence either the
  left side of (\ref{eq:5}) factors completely or it is irreducible.
  In the second case $\alpha$ is of degree $p$ over $K$ and
  $K(\alpha)$ contains all the conjugates of $\alpha$ over $K$.  Therefore
  the extension $K(\alpha)/K$ is Galois of degree $p$, and hence
  cyclic.  
  
  Next consider the different of $\alpha$.  This different
  is a power of $a$.  By \cite[Lemma 2, page 71]{C}, this implies that no
  prime of $K$, not dividing $a$,  at which $\alpha$ is integral has any ramified factors
  in the extension $K(\alpha)/K$.  Suppose now ${\aaa_K}$ is a prime of
  $K$ described in the statement of the lemma.  Let ${\aaa_{K(\alpha)}}$ be a $K(\alpha)$-prime above ${\aaa_K}$.  Then $\ord_{\aaa_{K(\alpha)}}\alpha =0$ and 
  \[
(p-1)\ord_{\aaa_{K(\alpha)}}a =\ord_{ \aaa_{K(\alpha)}}(\alpha^p-1)= \ord_{ \aaa_{K(\alpha)}}(\alpha-1)^p\equiv 0 \mod p.
\]
  Thus,
  ${\aaa_{K(\alpha)}}$ must be totally ramified over ${\aaa_K}$. 
  
  To see that the poles of $a$ split in the extension, consider the minimal polynomial of $\beta =\frac{\alpha}{a}$.    We have that 
\[
\frac{\alpha^p}{a^p}-\frac{\alpha}{a}-\frac{1}{a^p}=0
\]
 or $\beta$ is a root of 
\[
T^p-T-\frac{1}{a^p}=0.
\]
  The coefficients of this polynomial are integral at poles of $a$, its discriminant is a constant and modulo any pole of $a$ the polynomial splits completely.

\end{proof}
\begin{proposition}
\label{prop:close}
Let $a \in K\setminus \{0\}$, $\pp_K$-- a prime of $K, \alpha \in \tilde K$ a root of the  polynomial 
\begin{equation}
\label{eq:d}
T^p-a^{p-1}T-1=0,
\end{equation}
 where $\ord_{\pp_K}a >0$.  Let $z \in K$ be such that $\ord_{\pp_K}z=0$ and for some $x_0 \in K, x_1 \in K \setminus\{0\}$, integral at $\pp_K$,  we have that
\[
\ord_{\pp_K}(x_0^p -a^{p-1}x_0x_1^{p-1} +x_1^p - z) > 2(p-1)\ord_{\pp_K}(ax_1).
\]
  In this case, for some $y \in K_{\pp_K}(\alpha)$ we have that ${\mathbf N}_{K_{\pp_K}(\alpha)/K_{\pp_K}}(y)=z$.
\end{proposition}
\begin{proof}
Observe that by Lemma \ref{lemma:allsol} we have that $\alpha \in K$ or of degree $p$ over $K$.  A similar argument applies to $K_{\pp_K}$.  So without loss of generality we can assume that $\alpha$ is of degree $p$ over $K_{\pp_K}$.   Next we note that ${\mathbf N}_{K_{\pp_K}(\alpha)/K_{\pp_K}}(\alpha)=1$ and if we let $\alpha_i=\alpha+ia, i=0,\ldots,p-1$, then $\sum_{i_1<i_2<\ldots <i_{p-1}}\alpha_{i_1}\ldots\alpha_{i_{p-1}}=-a^{p-1}$.  All other basic symmetric functions of the conjugates are zero.

 We will seek an element of $K_{\pp_K}(\alpha)$ of the form $x_0+x_1\alpha$ to have the norm equal to $z$.  In other words we want to solve 
\begin{equation}
\label{eq:P1}
P(x_0,x_1,z)={\mathbf N}_{K_{\pp_K}(\alpha)/K_{\pp_K}}(x_0+x_1\alpha)-z=0
\end{equation}
in $K_{\pp_K}$.  We rewrite \eqref{eq:P1} as
\[
P(x_0,x_1,z)={\mathbf N}_{K_{\pp_K}(\alpha)/K_{\pp_K}}(x_0+x_1\alpha)-z=\prod_{i=0}^{p-1}(x_0+x_1\alpha_i)-z= 
\]
\[
x_0^p +x_0^{p-1}x_1\left (\sum_i \alpha_i\right) + x_0^{p-2}x_1^2\left (\sum_{i\not=j} \alpha_i\alpha_j \right )+ \ldots 
\]
\[
+x_0x_1^{p-1}\left (\sum_{i_1<i_2<\ldots<i_{p-1}}\alpha_{i_1}\ldots\alpha_{i_{p-1}}\right ) + x_1^p\left (\prod_i \alpha_i \right )-z=
\]
\[
x_0^p -a^{p-1}x_0x_1^{p-1} +x_1^p - z.
\]
Further,
\[
\partial P(x_0,x_1,z)/\partial x_0=a^{p-1}x_1^{p-1} \not=0.
\]
Therefore, given our assumptions,  Hensel's lemma guarantees us a solution in $K_{\pp_K}$.
\end{proof}
\begin{corollary}
\label{big enough}
Under assumptions of Proposition \ref{prop:close}, suppose $\ord_{\qq_K}z \equiv 0 \mod p$ for all primes $\qq_K$ of $K$ not in the zero divisor of $a$, and $\ord_{\pp_K}(z-1) > 2(p-1)\ord_{\pp_K}a$ for all $\pp_K$ occurring in the zero divisor of $a$.  In this case there exists  $y \in K(\alpha)$ such that ${\mathbf N}_{K(\alpha)/K}(y)=z$.
\end{corollary}
\begin{proof}
We use Hasse norm principle as we did in Proposition \ref{prop:norm}.  The only difference is that we now have primes ramifying in the extension.  For these ramifying primes we apply Proposition \ref{prop:close} with $x_0=0$ and $x_1=1$.

\end{proof}

We now use Corollary \ref{big enough} to set up a test for integrality when only extensions of degree $p$ are available.  
\begin{proposition}
\label{eq:ordp}
Let $a, b, x \in K$, $1 + a^{\ell}(bx^p+b^p) \not =0$ for all $\ell >\ell_0$.  Let $G=K(\delta)$ where $\delta=\delta(\ell)$  is a root of the polynomial $T^p+T+\frac{1}{1 + a^{\ell}(bx^p+b^p)}$ for $\ell >\ell_0$. If $\alpha \in \tilde K$ is a root of \eqref{eq:d},  $[K(\alpha):K]=p$, then for all sufficiently large $\ell >\ell_0$ the norm equation
\begin{equation}
\label{eq:normp}
{\mathbf N}_{G(\alpha)/G}(y)=1+a^{\ell}(bx^p+b^p)
\end{equation}
has a solution $y \in G(\alpha)$ if and only if for every $\pp_K$ occurring in the pole divisor of $b$ and not in the divisor of $a$ one of the following alternatives hold:
\be
\item \label{option1} $\pp_K$ splits in the extension $K(\alpha)/K$;
\item \label{option2}  $(\ord_{\pp_K}b  \equiv 0 \mod p) \lor (\ord_{\pp_K}(bx^p) > \ord_{\pp_K}b^p)$
\ee

\end{proposition}
\begin{proof}
First of  let $\pp_K$ be a prime of $K$ such that neither alternative holds.  In this case  
\[
\ord_{\pp_K}(bx^p) < \ord_{\pp_K}b^p <0
\]
 and since $\ord_{\pp_K}a =0$ by assumption, we conclude that 
\begin{equation}
\label{eq:notequivp}
(\ord_{\pp_K}(1+a^{\ell}(bx^p+b^p))<0) \land (\ord_{\pp_K} bx^p \not \equiv 0 \mod p) \land (\ord_{\pp_K}(1+a^{\ell}(bx^p+b^p)) \not \equiv 0 \mod p).
\end{equation}
  Now note that by Lemma \ref{lemma:allsol}, for any $\ell$ it is the case that $\pp_K$ splits completely in the extension $G/K$ and therefore in the extension $G(\alpha)/G$ no factor $\pp_{G}$ of $\pp_K$ splits.  Further, from \eqref{eq:notequivp} for every such factor we have
\begin{equation}
\label{eq:notequivpup}
(\ord_{\pp_G}(1+a^{\ell}(bx^p+b^p))<0) \land (\ord_{\pp_G}(1+a^{\ell}(bx^p+b^p)) \not \equiv 0 \mod p).
\end{equation}
Hence the norm equation cannot have a solution.

If for some pole  $\pp_K$ of $b$, not occurring in the divisor of $a$,  alternative \eqref{option1} holds, then we have that \eqref{eq:d} has a solution in the residue field of $\pp_K$.  Further, no factor of $\pp_K$ in $G$ will be ramified in the extension $G(\alpha)/G$ and the residue field of any factor of $\pp_K$ is an extension of the residue field of $\pp_K$.  Hence \eqref{eq:d} will have a root in that field too.  Thus, every factor of $\pp_K$ in $G$ will split completely in the extension $G(\alpha)/G$.  Hence the norm equation will be solvable locally at all the factors of $\pp_K$.  

If for some pole  $\pp_K$ of $b$, not occurring in the divisor of $a$,  alternative \eqref{option2} holds, then $\ord_{\pp_K}(1+a^{\ell}(bx^p+b^p)) \equiv 0 \mod p$.

Additionally, if $\aaa_K$ is any prime of $K$ such that 
\[
[\ord_{\aaa_K}(1+a^{\ell}(bx^p+b^p)) >0] \land [\ord_{\aaa_K}(1+a^{\ell}(bx^p+b^p)) \not \equiv 0 \mod p],
\]
then $\aaa_K$ is ramified completely in the extension $G/K$.  Hence in $G$ all zeros of $1+a^{\ell}(bx^p+b^p)$ are of order divisible by $p$.  Finally for all sufficiently large $\ell$, for any $\ttt_K$  occurring in the zero divisor of $a$, we have 
\[
\ord_{\ttt_K}a^{\ell}(bx^p+b^p) > 2(p-1)\ord_{\ttt_K}a.
\]
Thus by Corollary \ref{big enough}, the norm equation will have a solution.
\end{proof}
In the same fashion one can prove another version of this proposition.
\begin{proposition}
\label{prop:need later}
Let $x, Y \in K$, with $1+a^{\ell}Y \not =0$ for all $\ell >\ell_0$.  If $\alpha \in \tilde K$ is a root of \eqref{eq:d},  $[K(\alpha):K]=p$,   $H=K(\beta)$, where $\beta=\beta(\ell)$ is a root of $T^p+T+\frac{1}{1 + a^{\ell}Y}$ for some $\ell > \ell_0$, then for all sufficiently large $\ell >\ell_0$ the norm equation
\begin{equation}
\label{eq:normp}
{\mathbf N}_{H(\alpha)/H}(y)=1+a^{\ell}Y
\end{equation}
has a solution $y \in H(\alpha)$ if and only if for every $\pp_K$ occurring in the pole divisor of $Y$ and not in the divisor of $a$ we have either $\ord_{\pp_K}Y \equiv 0 \mod p$ or $\pp_K$ splits completely in the $K(\alpha)/K$.
\end{proposition}
\section{Defining rings of $\calS_K$ integers uniformly in $p$ and $K$ when the field of constants has an extension of degree not divisible by $p$.}
\label{divorder}
\setcounter{equation}{0}

We now consider the first version of a definition of the ring of $\calS_K$-integers under Assumption \ref{ass:pnotq}.

\begin{proposition}
\label{def1}
If $\calS_K$ is a finite set of primes of  $K$ then 
\begin{equation}
\label{eq:A}
\begin{array}{c}
O_{K, \calS_K} \setminus \{0\} =\\
\{x \in K|\forall c \in \Phi(K,\calS_K) \setminus \{0\} \forall b \in K \mbox{ with }  (bx^q+b^q \not = 0 ) \exists y \in L_3(\sqrt[q]{c}):
 {\mathbf N}_{L_3(\sqrt[q]{c})/L_3}(y) = bx^q+b^q\}
 \end{array}
\end{equation}
\end{proposition}
\begin{proof}
Denote by $A$ the set defined by the right-side of  \eqref{eq:A}.  Observe now that if  $\pp_K \not \in \calS_K$, then by the Strong Approximation Theorem, there exists $c \in K$ such that $c$ is not a $q$-th power mod $\pp_K$ and $c \in   \Phi(K, \calS_K)$.  Next assuming $\ord_{\pp_K}x <0$, let $b \in K$ be such that  $\ord_{\pp_K}b=-1$, and note that by Proposition \ref{prop:norm} we have that $x \not \in A$.  The fact that any element of $O_{K, \calS_K}$ belongs to $A$ also follows directly from Proposition \ref{prop:norm}.
\end{proof}
If we let $\calS_K$ be empty, then we define the set of elements of $K$ without any poles, i. e. constants.
\begin{corollary}
The set of constants of $K$ has the following definition:
\[
\{x \in K|\forall c  \in K\setminus \{0\}  \forall b \in K  \mbox{ with }  (bx^q+b^q \not = 0 ) \exists y \in L_3(\sqrt[q]{c}): {\mathbf N}_{L_3}(\sqrt[q]{c})/L_3(y) = bx^q+b^q\}.
\]
\end{corollary}

We now consider definitions of integrality at finitely many primes to define elements of $K$ contained in $\Phi(K, \calS_K)$.  
\begin{proposition}
\label{finitely many for p not q}
Let $\qq_K \not=\pp_K$ be $K$-primes.  Let  $a_{\pp_K}$ be an element of $K$ of order $-1$ at $\pp_K$ and with only one other pole, at $\qq_K$.  Assume $b_{\pp_K} \in K$ is such that it is not a $q$-th power modulo  $\pp_K$ and is equivalent to 1 modulo $\qq_K$. 
  Now let 
\[
L_4=K(\sqrt[q]{1+ a_{\pp_K}^{-1}},\sqrt[q]{1+ (a_{\pp_K}x^q+a_{\pp_K}^q)^{-1}},\sqrt[q]{1+(b_{\pp_K}+b_{\pp_K}^{-1})a_{\pp_K}^{-1}})
\]
and
\[
B(K,b_{\pp_K},a_{\pp_K})=
\{x \in K|\exists y \in L_4(\sqrt[q]{b_{\pp_K}}): {\mathbf N}_{L_4(\sqrt[q]{b_{\pp_K}})/L_4}(y) = a_{\pp_K}x^q+a_{\pp_K}^q\}.
\]
We claim $B(K,b_{\pp_K},a_{\pp_K})=\{x \in K|\ord_{\pp_K}x \geq 0\}$.
\end{proposition}
\begin{remark}
Note that $a_{\pp_K}$ and $b_{\pp_K}$ as specified above exist by the Strong Approximation Theorem.
\end{remark}
\begin{proof}
The proof of the proposition is almost identical to the proof of Proposition \ref{prop:norm}.  
\be
\item By construction no pole of $a_{\pp_K}$ in $K$ occurs in the divisor of $b_{\pp_K}$, since $b_{\pp_K} \equiv 1 \mod \qq_K$  and is not a $q$-th power  modulo $\pp_K$.  Thus, $(b_{\pp_K}+b_{\pp_K}^{-1})a_{\pp_K}^{-1}$ has poles at all the primes occurring in the divisor of $b_{\pp_K}$.  Also, all zeros of $a_{\pp_K}$ of orders not divisible by $q$ in $K$ are ramified with ramification degree $q$ before  we adjoin $\sqrt[q]{1+(b_{\pp_K}+b_{\pp_K}^{-1})a_{\pp_K}^{-1}}$, and therefore in $L_4$ all zeros and poles of $b_{\pp_K}$ have order divisible by $q$.
 
\item  We have that  
\[
\ord_{\pp_K}(a_{\pp_K}x^q) \not = \ord_{\pp_K}(a_{\pp_K}^q),
\]
since the left order is not equivalent to 0 mod $q$ and the right one is.  Thus under these circumstances, $\ord_{\pp_K}(a_{\pp_K}x^q+a_{\pp_K}^q) \equiv 0 \mod q$, implies that 
\[
\ord_{\pp_K}(a_{\pp_K}x^q+a_{\pp_K}^q)=\ord_{\pp_K}(a_{\pp_K}^q)
\]
 and 
\[
 \ord_{\pp_K}x > \frac{q-1}{q} \ord_{\pp_K}a_{\pp_K} > \ord_{\pp_K}a_{\pp_K}=-1. 
 \]
  Of course, the  inequality cannot hold if $\ord_{\pp_K}x <0$.
  \ee
\end{proof}
\begin{remark}
\label{rem:existential}
We note that the complement of $B(K,a_{\pp_K},b_{\pp_K})$ is also existentially definable: 
\[
\overline{B(K,a_{\pp_K},b_{\pp_K})}=\{x \in K\setminus \{0\}|\frac{a_{\pp_K}}{x} \in B(K,a_{\pp_K},b_{\pp_K})\}.
\]
Thus, we have an existential definition for  $\Phi(K,\calS_K)$ and its complement with parameters in $K$: we need an element $a_{\pp_K}$, as described above for each $\pp_K$, but by the Strong Approximation Theorem, we can find a $b_{\pp_K}=b_{\calS_K}$ to work for all $\pp_K \in \calS_K$.
\end{remark}
The discussion above can be summarized in the following corollary.
\begin{corollary}
\label{diffversion}
Let $L_3, a_{\pp_K}, \pp_K \in \calS_K, b_{\calS_K}$ be as described above.  In this case
 \[
 x \in O_{K,\calS_K} \setminus \{0\}
 \]
 \[
 \Updownarrow
 \]
 \[
\forall c\in K \setminus \{0\} \forall b \in K  \mbox{ with }  (bx^q+b^q \not = 0 ) 
\]
\begin{equation}
\label{eq:Sint}
 ( \lor_{\pp_K \in \calS_K} ((c-1)\not \in B(K,a_{\pp_K},b_{\calS_K}) ) \lor  \exists y \in L_3(\sqrt[q]{c}):{\mathbf N}_{L_3(\sqrt[q]{c})/L_3}(y) = bx^q+b^q).
\end{equation}
 
\end{corollary}
The definability aspects of Corollary \ref{diffversion} can be restated as the following theorem.
\begin{theorem}
\label{S-integers}
If $K$ is a function field of characteristic $p>0$,  $q$ is a prime number different from $p$, $\calS_K$ is a finite set of primes of $K$, and the field of constants $\G_p$ of $K$ satisfies the following conditions:
\be
\item $ \G_p$ is algebraic over a finite field,
\item $ \G_p$ has an extension of degree $q$,
\item $ \G_p$ contains a primitive $q$-th root of unity,
\ee
then the ring of $\calS_K$-integers has a $\forall\forall \exists \ldots \exists$-definition over $K$ requiring $|\calS_K|+1$ parameters and this definition is uniform across all the fields satisfying the conditions above.
Further the field of constants has a  $\forall\forall \exists \ldots \exists$-definition over $K$ requiring no parameters  and this definition is uniform across all the fields satisfying the conditions above.  
\end{theorem}

\begin{proof}
The only point that needs to be clarified here is how to rewrite the norm equation 
\begin{equation}
{\mathbf N}_{L_3(\sqrt[q]{c})/L_3}(y) = bx^q+b^q
\end{equation} 
in a polynomial form with variables taking values in $K$.   We start with rewriting the norm equation itself.   If $M$ is any field of characteristic $p\not =q$ containing $\xi_q$ and $c \in M \setminus M^q$, $u_1,\ldots, u_{q}, z \in M$, $y = \sum_{i=1}^{q}a_i\sqrt[q]{c}^{(i-1)}$, then
\begin{equation}
\label{eq:normrewrite}
{\mathbf N}_{M(\sqrt[q]{c})/M}(y)-z =\prod_{j=0}^{q-1}\sum_{i=1}^{q}u_i\xi_q^{(i-1)j}\sqrt[q]{c}^{(i-1)}-z=N(\xi_q,u_1,\ldots,u_{q},c,z) \in \F_p[\xi_q, U_1,\ldots,U_{q}, C, Z],
\end{equation}
and the coefficients of $N(X, U_1,\ldots,U_{q}, C, Z)$ depend on $q$ only.  While the specific value of $\xi_q$ depends on $p$, we  can replace $\xi_q$ by a variable $\xi$ such  that a system of equations in the variables $\xi, y_1,\ldots, y_{q-1}$ can be satisfied over $G$: 
\[
\left \{
\begin{array}{c}
(\xi-1)y_1=1\\
(\xi^2-1)y_2=1\\
\ldots\\
(\xi^{q-1}-1)y_{q-1}=1\\
\xi^q=1
\end{array}
\right .
\]
We denote the system above together with $N(\xi,u_1,\ldots,u_{q},c,z)$ by $\bar N(\xi,u_1,\ldots,u_{q},c,z)$.
If $c, w \in M, c=w^q$, then for any $z \in F$ the equation $\bar N(\xi, U_1,\ldots,U_q,c,z)=0$ has solutions $a_1,\ldots, a_{q} \in M(\xi_q)$.
Indeed, consider the following system of equations:
\[
\left \{
\begin{array}{c}
\sum_{i=0}^{q-1}a_iw^{i}=z,\\
\sum_{i=0}^{q-1}a_i\xi_q^{ij}w^{i}=1, j=1,\ldots,q-1
\end{array}
\right .
\]
This is a nonsingular system with a matrix $(\xi_q^{ij}w^{i}), i=0,\ldots, q-1,j = 0,\ldots,q-1$ having all of its entries in $M(\xi_q)=M$.  Since the vector $(z,1,\ldots,1)$ also has all of its entries in $M$, we conclude that the system has a unique solution in $M$.
So if we, for example, consider ${\mathbf N}_{L_3(\sqrt[q]{c})/L_3}(y) = bx^q+b^q$ with  potential solutions $y$ ranging over $L_3(\sqrt[q]{c})$, then we can conclude that that this norm equation is equivalent to a polynomial equation 
\begin{equation}
\label{stepdown}
\bar N(\xi,u_1,\ldots,u_{q},c,bx^q+b^q) =0
\end{equation}
 with coefficients in $ \G_p$ and potential solutions 
 \[
 u_1,\ldots, u_q \in L_3=L_2(\sqrt[q]{1+(c+c^{-1})x^{-1}}).  
\]
We now would like to replace \eqref{stepdown} by an equivalent equation but with solutions in $L_2$.  We have to consider two options: either there exists $\gamma \in L_2$ such that 
\begin{equation}
\label{eq:gamma}
\gamma^q=1+(c+c^{-1})x^{-1}
\end{equation}
 and in this case all the solutions $u_1,\ldots, u_q \in L_2$, or $1+(c+c^{-1})x^{-1}$ is not a $q$-th power in $L_2$, so that $u_i=\sum_{j=0}^{q-1}u_{i,j}\gamma^j$, where $\gamma$ is as in \eqref{eq:gamma} and $u_{i,j} \in L_2$.  In the latter case we can rewrite \eqref{stepdown} first as 
\begin{equation}
\label{eq:system}
N(\sum_{j=0}^{q-1}u_{1,j}\gamma^j,\ldots,\sum_{j=0}^{q-1}u_{q,j}\gamma^j,c,bx^q+b^q) =0,
\end{equation}
 and then as a system of equations over $L_2$ using the fact that the first $q$ powers of $\gamma$ are linearly independent over $L_2$.   Thus for any $c, b, x \in K$ we can conclude that \eqref{stepdown} has solutions $u_1,\ldots, u_q \in L_3$ if and only if either there exists $\gamma \in L_2$ satisfying \eqref{eq:gamma} and there exists $u_1,\ldots, u_q \in L_2$ satisfying \eqref{stepdown} or there exist $u_{1,0}, \ldots, u_{q, q-1} \in  L_2$ satisfying \eqref{eq:system} rewritten as a system of equations with coefficients in  $L_2$ under the assumption that $1,\gamma, \ldots, \gamma^{q-1}$ are linearly independent over $ L_2$ and $\gamma$ satisfies  \eqref{eq:gamma}.  

Note that if \eqref{stepdown} has solutions in $L_3$ and $\gamma \in L_2$ (so that $L_2=L_3$), we can still find solutions to \eqref{eq:system}  using the system of equations obtained from rewriting \eqref{eq:system} under the assumption that powers of $\gamma$ are linearly independent.  This is so because the resulting system is equivalent to \eqref{eq:system} modulo $\gamma^q-1-(c+c^{-1})x^{-1}$.  So any solution of this system obtained with values of $x, c \not=0$ and $\gamma$ satisfying \eqref{eq:gamma}  will produce a solution to \eqref{eq:system}.  Further, if $\gamma \in L_2$, then \eqref{eq:system} has solutions with $u_{i,0}=u_i$ and  $u_{i,j}=0$ for $j>0$ and these solutions will remain solutions of \eqref{eq:system} rewritten as a system.  Thus, whether or not $\gamma \in L_2$, we can replace $\eqref{eq:system}$ by the system obtained by treating $\{1,\gamma, \ldots,\gamma^{q-1}\}$ as being linearly independent and thus construct a system of equations having solutions in $L_2$ if and only if \eqref{stepdown} had solutions in $L_3$.

  We can clearly can continue in this manner until we reach $K$.
 
  For a more general and formal discussion of the rewriting techniques we refer the reader to the section on coordinate polynomials in \cite{Sh34}. 
\end{proof}

\begin{notation}
 \label{not:S-integers}
   Extension $L_3(\sqrt[q]{c})$ is potentially of degree $q^4$ over $K$, and so  the element $y$ of $L_3(\sqrt[q]{c})$ will be represented by a tuple of $q^4$ elements of $K$. We denote this tuple of $K$-elements by $\bar y$.  For future reference, denote  \eqref{eq:Sint} by 
   \[
   S(x,\bar y,b,c, b_{\calS_K},\{a_{\pp_K}\}_{\pp_K \in \calS_K}).
   \]  
\end{notation}   
   We can specialize Theorem \ref{S-integers} in several ways.  For example, if we restrict $K$ to being a global field, we can get a version of Rumely-type result.
\begin{corollary}
If $K$ is a global field of characteristic greater than 2 or of characteristic equal to 2 with the constant field containing subfield of size 4, and $\calS_K$ is a finite set of primes of $K$, then the ring of $\calS_K$-integers of $K$ has a uniform definition over $K$ of the form $\forall\forall \exists \ldots \exists$ using $|\calS_K|+1$ parameters.
\end{corollary}

One can also get rid of the requirement that the fields contain the relevant roots of unity by using more quantifiers. If we wanted to use the $q$-version of the definition and the field did not have a primitive $q$-th root of unity, then we would have to draw the ``$c$'' in the norm equation from an extension of degree less or equal to $q-1$.  So to capture such a ``$c$'' in the extended field, we would potentially need to use $q$ variables instead of $1$, adding $q-1$ quantifiers.

  This discussion leaves out the fields whose constant fields have only extensions with the degree equal to the power of the characteristic.  We will be able to handle that case also but unfortunately in a very non-uniform manner, i.e. the definition not only will depend on the characteristic but also on the size of the constant field relative to the genus of the function field.

\section{Updating some old results}
\setcounter{equation}{0}
As we have mentioned in the introduction, the results of this section are necessary to produce a uniform definition of $p$-th powers over the field.  We start with a uniform definition of $p$-th powers over some rings of $\calS$-integers.  Below we require additional notation.
\begin{notation}
Given a non-zero element $x \in K$, let $h_K(x)$ denote the $K$-height of $x$, i.e. the degree of the zero divisor or the degree of the pole divisor of $x$ in $K$.  For $x=0$ set $h_K( x)=0$.
\end{notation}
\subsection{$p$-th powers over the ring where only one prime is allowed in the denominator}
\begin{lemma}[Lemma 2.12 and Lemma 4.5 of \cite{Sh31}]
\label{uniformpowers}
 Let $\pp_K$ be a prime of $K$ and suppose $w \in O_{K,\{\pp_K\}}$ is not a unit.  If $p > 2$, $f \in O_{K,\{\pp_K\}}$, the following equations are satisfied in all the other variables with values in $O_{K,\{\pp_K\}}$, then for some $k \in \Z_{>0}$ we have that $f=w^{p^k}$.
\begin{equation}
\label{eq:p}
\left \{
\begin{array}{c}
(f + 1)^2 - w(w + 2)h^2 = 1,\\
(2f + 1)^2 - 2w(w + 2)g^2 = 1.
\end{array}
\right .
\end{equation}

For characteristic $p = 2, f, g \in O_{K,\{\pp_K\}}$ the system
\begin{equation}
\label{eq2}
\left \{
\begin{array}{c}
f^2 + fhw^2 + h^2w^2 = w^2,\\
(f^2 + f)^2 + (w + 1)^2w^2(f^2 + f)g + w^2(w + 1)^2g^2 = (w + 1)^2w^2
\end{array}
\right .
\end{equation}
has solutions in $O_{K,\{\pp_K\}}$ if and only if $\exists  k \in \Z_{>0}$ such that $f = w^{2^k}$. 
\end{lemma}

We often need ``synchronized'' $p$-th powers as delivered by the following two lemmas.

For future reference denote \eqref{eq:p} or \eqref{eq2} (depending on the characteristic) by $PPE(f,w,h,g)$.

\begin{lemma}
\label{synch}
  Let $w, x, y=w(x+1)x+1 \in O_{K,\{\pp_K\}} \setminus \G_p$,  and assume the following statement is true. 
\begin{equation}
\label{sync:1}
\begin{array}{c}
\exists f_1,h_1,g_1, f_2,h_2,g_2,f_3,h_3,g_3  \in O_{K,\{\pp_K\}}:\\
PPE(f_1,y,h_1,g_1) \land PPE(f_2,yx,h_2,g_2) \land PPE(f_3,y(x+1),h_3,g_3) \land f_3=f_2 +f_1.
\end{array}
\end{equation}
In this case there exists $s\geq0$ such that $f_1=y^{p^s}, f_2=(yx)^{p^s}$, and $h=\frac{(f_1-1)f_1^2}{f_2(f_2+f_1)}=w^{p^s}$.
\end{lemma}
\begin{proof}
First of all observe that $y=wx(x+1) +1$ has no zeros in common with $x$ and $x+1$ since $w$ and $x$ have a pole at $\pp_K$ only.  Next note that  for some $s,r,m \in \Z_{\geq0}$  we have that $f_1=y^{p^s}, f_2=y^{p^r}x^{p^r}, f_3 =y^{p^m}(x+1)^{p^m}$
\begin{equation}
\label{sync:3}
f_3-f_1=y^{p^m}x^{p^m}+y^{p^m}-y^{p^s}=y^{p^r}x^{p^r}. 
\end{equation} 
If $p^s|\ord_{\qq_K}y| < p^m|\ord_{\qq_K}y| + p^m|\ord_{\qq_K}x|$, then $r=m$ and we are done.  So suppose $p^s|\ord_{\qq_K}y| \geq  p^m|\ord_{\qq_K}y |+ p^m|\ord_{\qq_K}x|$  and in particular $s > m$.  Since $(y,x(x+1))=1$, we now conclude from looking at zeros of $y$, that $r=m$ and we  are done.
\end{proof}
Again, for future reference, denote \eqref{sync:1}  by $SPPE(w, x, h, f,\ldots)$ so that if 
\[
SPPE(w, x, h, f,\ldots)
\]
 is satisfied over $O_{K,\{\pp_K\}}$, we have $h=w^{p^s},f =x^{p^s}$ for some non-negative integer $s$.

We will also need to specify that the exponent of $p$  is divisible by a fixed positive integer $z$ using a parameter over $O_{K,\{\pp_K\}}$.  
\begin{lemma}
\label{le:divbyz}
Let $g, g^{p^z} \in O_{K,\{\pp_K\}} \setminus \G_p$ be given.  Let $h, x \in O_{K,\{\pp_K\}}$ and assume the following equations are satisfied over $O_{K,\{\pp_K\}}$.
\begin{equation}
\label{same1}
PPE(g,h, \ldots)
\end{equation}
\begin{equation}
\label{same2}
SPPE(g,x,h,v, \ldots) 
\end{equation}

\begin{equation}
\label{same3}
(g^{p^z-1}-1) {\Large |_{O_{K,\{\pp_K\}}}} (\frac{h}{g}-1)
\end{equation}
In this case for some positive integer $s$ we have that $v =x^{p^s}$ and $s \equiv 0 \mod z$.
\end{lemma}
\begin{proof}
From \eqref{same1} and \eqref{same3} we have for some positive integer $s$ that $h=g^{p^s}$ and
\[
(g^{p^z-1}-1) {\Large |_{O_{K,\{\pp_K\}}}} (g^{p^s-1}-1).  
\]
Note that 
\[
\frac{g^{p^s-1}-1}{g^{p^z-1}-1} \in  \G_p(g) \cap O_{K,\{\pp_K\}}=\G_p[g]
\]
 and therefore we can consider the divisibility condition over the polynomial ring $\G_p[g]$ instead of $O_{K,\{\pp_K\}}$.  In this case it is easy to see that $(p^z-1) | (p^s-1)$ in $\Z$ and therefore $z |s$.  Now from \eqref{same2} we  derive $v=x^{p^s}$. 
\end{proof}
We  also need same power and power divisible by a fixed integer sets defined over a  function field.  Thus we also need to use the following result concerning (non-uniform) existential definability of $p$-th powers over a function field of positive characteristic.
\begin{proposition}
\label{prop:pthpower}
For any function field $K$ of positive characteristic  the set $P(K)=\{(f,g) \in K^2|  \exists s \in \Z_{\geq 0}, \, g=f^{p^s}\}$ is existentially definable (has a Diophantine definition) over $K$. (See \cite{ES2}.)
\end{proposition}
We denote the equations defining $P(K)$ by PPEF$(f,g)$ ($p$-th power equations over a field), so that PPEF$(f,g)$ hold over $K$ if and only if $\exists s \in \Z_{\geq 0}: g=f^{p^s}$.
\begin{lemma}
\label{prop:pthpower2}
If  $K$ is any function field, $x \in K$ any non-zero element, then the set
\[
\{(x, X, f ,Y) \in K^3|  \exists s \in \Z_{\geq 0}, \, X=x^{p^s}, Y=f^{p^s}\},
\]
and for a fixed $z \in \Z_{>0}$ the set
\[
\{(t,T, T_z) \in K^3|  \exists s \in \Z_{\geq 0}, \, t =T^{p^s}, T_z=t^{p^{zs}}\}
\]
are existentially (non-uniformly) definable (have a Diophantine definition) over $K$.
\end{lemma}
\begin{proof}
Let $t \not =0$ be a fixed element of $K$.  Without loss of generality we can assume $f \not =0$.  Assume initially that $K$ contains  $n \geq 2h_K(t)$  distinct  constants $c_0=0, c_1,\ldots, c_n$ and consider a system of equations:
\begin{equation}
\label{eq:p1}
\exists s_i, m_i, y_i \in \Z_{\geq 0}: Y_i=(f+c_i)^{p^{s_i}}, X_i=(x+c_i)^{p^{y_i}}, C_i=c_i^{p^{m_i}}, i=0,\ldots,n,  
\end{equation}
\begin{equation}
\label{eq:p2}
Y_i=Y_0 + C_i, X_i=X_0 +C_i, i=1,\ldots,n,
\end{equation}
\begin{equation}
\label{eq:p3}
\exists s \in \Z_{\geq 0}: T=t^{p^s}
\end{equation}
\begin{equation}
\label{eq:pn}
\exists \hat s_i \in \Z_{\geq 0}:Y_iT=(t(f+c_i))^{p^{\hat s_i}}, 
\end{equation}
\begin{equation}
\label{eq:pn2}
\exists \hat y_i \in \Z_{\geq 0}:X_iT=(t(x+c_i))^{p^{\hat y_i}}, 
\end{equation}
Assume \eqref{eq:p1}--\eqref{eq:pn2} are satisfied and first suppose $f$ is not a constant. Since $n \geq 2h_K(t)$, for some $i =0,\ldots, n$ it is the case that $f+c_i$ has a zero at a prime $\pp$ which does not occur in the divisor of $t$.
So from \eqref{eq:pn} we have for this $i$ that $(f+c_i)^{p^{s_i}}t^{p^{s}}=t^{p^{\hat s_i}}(f+c_i)^{p^{\hat s_i}}$ and comparing orders at $\pp$ we conclude that $\hat s_i=s_i$ leading to $t^{p^{s}}=t^{p^{\hat s_i}}$ and $\hat s_i=s$.  Further from \eqref{eq:p2} we have that $(f+c_i)^{p^{s_i}}=f^{p^{s_0}} +C_i$, where $c_i, C_i$ are constants. Thus $s_0=s_i=s$.  

If $f$ is a constant, then $Y_0$ is a constant and \eqref{eq:pn} for $i=0$ gives us $\hat s_0=s$.  We also conclude that $Y_0=f^{p^{s_0}}=f^{p^s}$.  By a similar argument we get $X_0=x^{p^s}$.

Conversely, if we set $m_i=s_i=\hat s_i=s=y_i=\hat y_i$, then all the equations can be satisfied. 
 
 Finally, if we need more constants than we have in our field $K$, we can initially take a separable constant field extension of $K$ (thus preserving the height of $t$), write down the necessary equations over the extension, and then rewrite the equations over $K$ so that all the coefficients of the equations are in $K$ and all the variables range over $K$.

 We can denote equations \eqref{eq:p1}--\eqref{eq:pn2} by SPPEF (same pth-power equations over a field).  So SPPEF$(x,X, f,Y)$ will hold if and only if $\exists s \in \Z_{\geq 0}: X=x^{p^s}$ and $Y=f^{p^s}$.
 
 Utilizing SPPEF equations, one can also produce $p^{zs}$-th powers of a given element for any fixed positive integer $z$.
For any $f \in K $ the following system of equations
\begin{equation}
\label{as}
\left \{
\begin{array}{c}
\mbox{SPPEF}(f, Y,x, X) \\
\mbox{SPPEF}(f,Y,X,X_1) \\
\ldots\\
\mbox{SPPEF}(f,Y,X_{z-2},X_{z-1}) \\
\end{array}
\right .
\end{equation}
has solutions in $K$ in all the remaining variables if and only if $X_{z-1}=f^{p^{zs}}$. 
\end{proof}
\subsection{Preparing to use $p$-th power equations  with the weak vertical method}
\begin{proposition}
\label{prop:pthpower3}
If $f \in K$ then for any $\pp_K$ with  $\ord_{\pp_K}t \geq 0$, and $e(\pp_K/\pp_{\F_p(t)})\not \equiv 0 \mod q$ there exists $\bar s \in \Z_{>0}$ such that for all positive integers $s \equiv 0 \mod \bar s$
\begin{equation}
\label{eq:positive}
\ord_{\pp_{K}}(t^{p^{\bar s}}-t) >0 \land \ord_{\pp_{K}}(t^{p^{\bar s}}-t) \not \equiv 0 \mod q.
\end{equation}
Additionally, if also $\ord_{\pp_K}f <0$, then
\begin{equation}
\label{eq:together}
\ord_{\pp_K}\frac{f^{qp^s}-f^q}{t^{p^s}-t} \not \equiv 0 \mod q.
\end{equation}
\end{proposition}
\begin{proof}
First of all, observe that in $\F_p(t)$ for any $\pp_{\F_p(t)}$ which is not a pole of $t$ there exists $\bar s \in \Z_{>0}$ such that  $\ord_{\pp_{\F_p(t)}}(t^{p^{\bar s}}-t)=1$. (This is true because every element of the algebraic closure of $\F_p$ must be a root of a polynomial $x^{p^s}-x$ for some $s$.)  Consequently it is also true that $\ord_{\pp_{\F_p(t)}}(t^{p^{s}}-t)=1$ for all $s \equiv 0 \mod \bar s$.  Therefore, assuming that the ramification degree of any factor of $\pp_{\F_p(t)}$ in $K$ is not divisible by $q$ we conclude that
\begin{equation}
\label{eq:t0}
\ord_{\pp_K}(t^{p^s}-t)\not \equiv 0 \mod q
\end{equation}
for any $\pp_K$ lying above $\pp_{\F_p(t)}$ in $K$ and any $s \equiv 0 \mod \bar s$.    Next suppose that $\ord_{\pp_K}f <0$ and observe that
\begin{equation}
\label{eq:f00}
\ord_{\pp_K}f^{qp^s}-f^q=qp^s\ord_{\pp_K}f \equiv 0 \mod q.
\end{equation}
 Finally combining \eqref{eq:t0} and \eqref{eq:f00}, we conclude that \eqref{eq:together} is true.
\end{proof}
Next we use a similar idea to require certain zeros.
\begin{lemma}
\label{le:ordminus1}
If $f \in K$, $s \in \Z_{>0}$ then for any $\pp_K$ not ramified over $\F_p(t)$ with $\ord_{\pp_{K}}t \geq 0$, $\ord_{\pp_{K}}f \geq 0$, and $\ord_{\pp_K} (t^{p^s}-t) >0$  either
\[
\ord_{\pp_K}\frac{f^{p^s}-f}{t^{p^s} -t}=-1 \mbox{ or } \ord_{\pp_K}(f^{p^s}-f)>0
\]
\end{lemma}
\begin{proof}
If for some $\pp_K$ as in the statement of the lemma we  do not have $\ord_{\pp_K}(f^{p^s}-f)>0$, then given our assumptions,  $\ord_{\pp_K}(f^{p^s}-f)=0$ and therefore
\[
\ord_{\pp_K}\frac{f^{p^s}-f}{t^{p^s} -t}= -\ord_{\pp_K}(t^{p^s} -t)=-1,
\]
where the last equality follows from the fact that $t^{p^s}-t$ does not have zeros of order higher than one in $\F_p(t)$ and we assumed $\pp_K$ is not ramified in the extension $K/\F_p(t)$.
\end{proof}
The first benefit of having zeros as described in the lemma above is explained below:
\begin{lemma}
\label{le:zerosbelow}
If $s \in \Z_{>0}$, $f \in K$ and $\pp_K$ is a prime of $K$ with $\ord_{\pp_K}(f^{p^s}-f) >0$, then for every $\pp_{\tilde K}$ lying above $\pp_K$ in $\tilde K$ there exists $b \in \tilde \F_p$ such  that $\ord_{\pp_{\tilde K}}(f -b)>0$.
\end{lemma}
\begin{proof}
Over $\tilde K$ we can factor $f^{p^s}-f=\prod_{b \in \F_{p^s}}(f - b)$.  Observe that no two factors in the product share a zero, and therefore  for every factor $\pp_{\tilde K}$ of $\pp_K$ in $\tilde K$ there must be a unique $b \in \tilde \F_p$ such that $\ord_{\pp_{\tilde K}}(f -b)>0$.
\end{proof}

Finally we state the property which makes sure that $p$-th power equations ``work'' for polynomials.
\begin{lemma}
\label{will work}
Let $f \in \F_{p^s}[t]$, where $t$ is transcendental over $\F_{p^s}$.  In this case 
\[
\frac{f^{p^s}-f}{t^{p^s}-t} \in \F_{p^s}[t].
\]
\end{lemma}
\begin{proof}
Observe that for any non-negative integer $i$ we have that 
\[
t^{ip^s}-t^i\equiv 0 \mod (t^{p^s}-t) \mbox{ in } \F_{p^s}[t].
\]
Let $f=\sum_{i=0}^{d}a_it^i, a_i \in \F_{p^s}$ and $a_i^{p^s}=a_i$.  Now 
\[
f^{p^s}=\sum_{i=0}^{d}a_it^{ip^s} \equiv  \sum_{i=0}^{d}a_it^{i} =f \mod  (t^{p^s}-t) \mbox{ in } \F_{p^s}[t].
\]
\end{proof}

\subsection{The weak vertical method, bounds on heights and prime splitting in cyclic extensions}
\label{subsec:weak}
We describe  what elsewhere we called ``the weak vertical method'' (see for example Chapter 10 of \cite{Sh34}).  It is a method of showing that an element belongs to a (rational) subfield of a given field.  First we need the following lemma.
\begin{lemma}
\label{basis}
Let $M/K$ be a finite separable extension of function fields.  Let $\sigma_1=\id,\ldots,\sigma_n$ be all the embeddings of $M$ into $\tilde K$ leaving $K$ fixed. Let $\{ \omega_1=1, \ldots, \omega_k \}$ be a basis of $M$ over
  $K$ and let $w \in M$ with $w=\sum_{i=1}^na_i\omega_i$.  Suppose now that $w, \omega_1, \ldots, \omega_n $ are all integral with respect to a prime $\pp_K$ (i.e. neither the listed elements nor their conjugates have any poles at factors of $\pp_K$),  and $\ord_{\pp_K}\det^2(\sigma_j(\omega_i))=0$.  In this case, $\ord_{\pp_K}a_i \geq 0$.
\end{lemma}
\begin{proof}
The lemma follows from considering the linear system 
\begin{equation}
\label{ls}
\sum a_i\sigma_j(\omega_i)=\sigma_j(w), i,j=1,\ldots,n
\end{equation}
 and solving it by Cramer's rule for $a_1,\ldots, a_n$.   
\end{proof}
\begin{proposition}[Slightly modified Theorem 10.1.1 of \cite{Sh34}]%
\label{prop:weak}
Let $M/K$ be a finite separable extension of function fields.  Let $\{ \omega_1=1, \ldots, \omega_n \}$ be a basis of $M$ over
  $K$ and let $w \in M$ with $w=\sum_{i=1}^na_i\omega_i$.  Suppose also that there exists a finite set $\calV_K$ of primes of $K$ such that
\begin{equation}%
\label{eq:equiv}
w\equiv b(\mathfrak A) \mod \mathfrak A,
\end{equation}%
where $\AA \in  \calV_K$, $b(\mathfrak A) \in K$ and we interpret the equivalence as saying that for any factor $\mathfrak c$ of $\mathfrak A$ in $M$ we
have that $\ord_{\mathfrak c}(w-b(\mathfrak A)) \geq e(\mathfrak c/\mathfrak A)$, the ramification degree of $\mathfrak c$ over $\mathfrak A$.  Assume that $w$ and all elements of the basis are integral at all the prime of $K$, $\det^2(\sigma_j(\omega_i)$ has no zeros at primes of $\calV_K$ and $|\calV_K| > \max h_K(a_i),  a_i \not = 0, i=2,\ldots,n$.  In this case $w \in K$.
\end{proposition}%
\begin{proof}%

   Observe $$w- b(\mathfrak A)= a_1 - b(\mathfrak A)+ a_2\omega_2 +\dots + a_n\omega_n.$$ 
For each prime ${\mathfrak A}$ of $\calV_K$, let
$B(\mathfrak A) \in K$ be such that $\ord_{\mathfrak A}B(\mathfrak A) =1$.  In this case 
\[
\displaystyle z=\frac{w-b(\mathfrak A)}{B(\mathfrak A)}
\]
 is integral
at $\mathfrak A$ and thus $z = \sum_{i=1}^n b_i(\Aa) \omega_i$, where
$b_i(\Aa) \in K$, is  integral at $\mathfrak A$.  We
observe further
\[
a_1 - b(\mathfrak A)+ a_2\omega_2 + \dots + a_n\omega_n=w-b(\mathfrak
A) = B(\mathfrak A)z=\sum_{i=1}^n B(\mathfrak A)b_i(\Aa)\omega_i.
\]
Thus, for $i=2,\ldots,n$ for all ${\mathfrak A}$ in $\calV_K$ we have
that $a_i = B(\mathfrak A)b_i(\Aa)$, implying that for $i=2,\ldots,n$,
for all ${\mathfrak A}$ in $\calV_K$ we have that $\ord_{\mathfrak A}a_i
>0$.  This is impossible unless for $i=2,\ldots,n$ we have that
$a_i=0$ and thus $w \in K$.
\end{proof}%
We now set up a mechanism to produce sets of primes $\calV_K$ satisfying the necessary conditions.  Below we describe the specific situation we encounter.
\be
\item Let $M/K$ be a finite Galois extension of degree $n$ of function fields over the same field $\G_p$ of constants.
\item Let $\sigma_1=\id,\ldots,\sigma_n$ be all the embeddings of $M$ into $\tilde K$ leaving $K$ fixed.
\item Let $\Omega =\{ \omega_1=1, \ldots, \omega_n\}$ be the basis of $M$ over $K$.
\item Let $f \in M$ and let $f=\sum_{i=1}^n a_i\omega_i, a_i \in K$.
\item \label{primeofK} If $\pp_K$ is a prime of $K$ such that $f, \omega_1,\ldots,\omega_n$ are integral with respect to $\pp_K$, i.e. neither $f, \omega_1,\ldots,\omega_n$ nor their conjugates over $K$ have poles at factors of $\pp_K$ in $M$, and $\ord_{\pp_K}\det^2(\sigma_j(\omega_i))=0$ then $\ord_{\pp_K}a_i \geq 0$ by Lemma \ref{basis}.   The number of primes $\pp_K$ dividing $\det^2(\sigma_j(\omega_i))$ or having factors occurring as poles of elements of $\Omega$ can be bound by $h_K(\det^2(\sigma_j(\omega_i)))\max_ih_M(\omega_i) \leq h_M(\det^2(\sigma_j(\omega_i)))\max_ih_M(\omega_i) $.
\item $h_M(a_i) \leq h_M(f) + C_M(\Omega)$, where $C_M(\Omega)$ depends on $n$, and $\max h_M(\omega_i)$ only:  this follows from linear system \eqref{ls}.   
\item \label{heights} Let $\hat M, \hat K$ be constant field extensions of $M$ and $K$ such that $\hat M$ and $\hat K$ have the same constant fields.  In this case $[\hat M:\hat K]=[M:K]$, $\Omega$ is a basis of $\hat M/\hat K$.  Further, $f \in M$, now as an element of $\hat M$ has the same coordinates with respect to $\Omega$ and both assertions of \eqref{primeofK} are still true since the degree of a divisor remains the same under a separable constant field extension.  For the same reason 
\[
h_{\hat M}(a_i) \leq h_{\hat M}(f) + C_{\hat M}(\Omega),
\]
 with $h_{\hat M}(a_i) =h_M(a_i)$,  $h_{\hat M}(f) =h_M(f)$, $C_{\hat M}(\Omega)=C_M(\Omega)$.
\item Let $t \in K$ with $K/\G_p(t)$ separable, and  let $\EE_M$ be the product of all primes of $M$ ramifying in the extension $M/\G_p(t)$,
\item Let $e_M$ be the degree of $\EE_M$.
\item Observe that $e_{\hat M}=e_M$.  This follows from the fact that if a prime does not ramify in the extension $M/\G_p(t)$, then a prime above it will not ramify in the extension $\hat M/\hat \G_p(t)$.  Thus under any algebraic constant field extension the number of such primes is bounded by the degree of $\EE_M$.
\ee

We now review an aspect of Chebotarev density theorem to get a sufficient number of degree one primes not splitting in a cyclic extension. 

\begin{lemma}
\label{Chebotarev}
Let $M_z/K_z$ be a degree $n$ cyclic extension of global fields over a finite field of constants $\F_{p^z}$ such that $M_z$ and $K_z$ have the same constant field.  Let $t \in K_z$ such that $M_z/\F_{p^z}(t)$ is Galois.  In this case for any sufficiently large integer $s \equiv 0 \mod z$, there exist  finite constant field extensions $\hat M=M_{s}$ of $M_z$ and $ \hat K=K_{s}$ of $K_z$ such that 
\be
\item $\hat M/\hat K$ is cyclic and $[\hat M: \hat K]=n$,
\item $\hat M$ and $\hat K$ have the same constant field of size $p^{s}$,
\item the number of degree 1 primes of $\hat K$ not splitting in $\hat M$ is greater than $p^{s/2}$,
\item each of this primes lies above the zero of  $t-a$, where $a \in  \F_{p^{s}}$.
\ee
\end{lemma}
\begin{proof}
The proof of this lemma follows from an effective version of Chebotarev Density Theorem and from the fact that the genus does not change under separable constant field extensions.  More precisely consider all primes of $M$ lying above degree one primes of $K$ not splitting in $M$.   For sufficiently large $s$ the number of such primes is greater $p^{s/2}$ by  Proposition 6.4.8 of \cite{Friednew}.  At the same time the relative degree of these primes of $K$ over the primes below in $\F_{p^{s}}(t)$ must be 1 and the primes below must be of degree 1, i.e. must either be zeros of elements of the form $t-a$ with $a^{p^{s}}-a=0$ or the pole of $t$.
\end{proof}
Finally, we show how $p$-th powers equations can be combined with the vertical method:
\begin{lemma}
\label{use order}
Let $z,s, M_z=M, K_z=K, \hat M, \hat K, \hat G_p, t$ be as in Lemma \ref{Chebotarev}, let $\EE_{M}$ be the product of all $M$-primes ramifying in the extension $M/\F_{p^{z}}(t)$, let $e_{M}$ be the degree of $\EE_{M}$, and let $f \in M$ be such that 
\[
\ord_{\pp_{M}}\frac{f^{p^{s}}-f}{t^{p^{s}}-t}\geq 0
\]
 for all $M$-primes $\pp_{M}$ that are not poles or zeros of $t$ and do not divide $\EE_{M}$.  Assume further that for some basis $\Omega =\{ \omega_1=1, \ldots, \omega_n\}$ of $M$ over $K$ we have that 
\[
p^{s/2}>h_{M}(f)+ C_{M}(\Omega) + h_{M}(\det(\sigma_j(\omega_i)))^2\max_ih_{M}(\omega_i) + e_{M} + h_{M}(t).
\]
In this case, $f \in K$.
\end{lemma}
\begin{proof}
Let $\calV_{\hat K}$ be the set of $\hat K$-primes such that  $\pp_{\hat K} \in \calV_{\hat K}$ if and only if $\deg \pp_{\hat K}=1$, $\pp_{\hat K}$ does not lie above a prime of $\F_{p^{s}}(t)$ ramifying in the extension $\hat M/\F_{p^{s}} (t)$, $\pp_{\hat K}$ is not a pole of $t$ and $\pp_{\hat K}$ does not split in the extension $\hat M/\hat K$.  By Lemma \ref{Chebotarev} and our assumptions, we can conclude that 
\[
|\calV_{\hat K}| >h_M(f)+ C(\Omega) + h_M(\det(\sigma_j(\omega_i)))^2\max_ih_M(\omega_i)  = 
\]
\[
h_{\hat M}(f)+ C_{\hat M}(\Omega) + h_{\hat M}(\det(\sigma_j(\omega_i)))^2\max_ih_{\hat{M}}(\omega_i) .
\]
For each $\pp_{\hat K}$ and some constant $b\in \hat K$ we have that $\ord_{\pp_{\hat K}}(f-b)>0$ and $|\calV_{\hat K}| > h_{\hat K}(a_i), i =2,\ldots,n$, where $f=\sum_{i=1}^na_i\omega_i, \omega_1=1$.  Thus, applying the weak vertical method (Proposition \ref{prop:weak}) we conclude that $f \in \hat K \cap M=K$.
\end{proof}
We now utilize the $p$-th power equations over a rational field.
\begin{lemma}
\label{rational}
Let $g \in \G_p(t)$ and assume that for some positive integer $s$ we have that the number of distinct primes dividing the pole divisor of $\displaystyle\frac{g^{p^s}-g}{t^{p^s}-t}$ is less than a positive constant $C$ while $h_{\G_p(t)}(g) < p^s -C$.  In this case $g \in \F_{p^s}(t)$.
\end{lemma}
\begin{proof}
Write $\displaystyle g=\frac{g_1}{g_2}, g_1, g_2 \in \G_p[t], (g_1,g_2)=1$.  Now 
\[
\frac{g^{p^s}-g}{t^{p^s}-t}=\frac{g_1^{p^s}g_2-g_1g_2^{p^s}}{g_2^{p^s+1}(t^{p^s}-t)}=\frac{(g_1^{p^s}-g_1)g_2 +g_1(g_2 -g_2^{p^s})}{g_2^{p^s+1}(t^{p^s}-t)}=
\]
\[
\frac{(g_1^{p^s}-\bar g_1)g_2+(\bar g_1-g_1)g_2 +g_1(g_2 -\bar g_2) +g_1(\bar g_2-g_2^{p^s})}{g_2^{p^s+1}(t^{p^s}-t)},
\]
where for a polynomial $z(t)=\sum_{i=0}^ma_it^i$, we let $\bar z=\sum_{i=0}^ma_i^{p^s}t^i$.
Thus,
\[
\frac{g^{p^s}-g}{t^{p^s}-t}=\frac{g_3+g_4}{g_2^{p^s}} + \frac{(\bar g_1-g_1)g_2 +g_1(g_2 -\bar g_2)}{g_2^{p^s+1}(t^{p^s}-t)},
\]
where $\displaystyle g_3=\frac{(g_1^{p^s}-\bar g_1)g_2}{(t^{p^s}-t)} \in \G_p[t]$ and $\displaystyle g_4=\frac{(g_2^{p^s}-\bar g_2)g_1}{(t^{p^s}-t)} \in \G_p[t]$.  Now observe that 
\[
\deg( (\bar g_1-g_1)g_2 +g_1(g_2 -\bar g_2) )\leq h_{\G_p}(g) < p^s -C.
\]
At the same time a factor of $t^{p^s}-t$ of degree greater than $p^s-C$ divides $(\bar g_1-g_1)g_2 +g_1(g_2 -\bar g_2)$.    Therefore, we must conclude that $(\bar g_1-g_1)g_2 +g_1(g_2 -\bar g_2)=0$ and $\displaystyle \frac{g_1}{g_2}=\frac{\bar g_1}{\bar g_2}$, where $\deg g_1=\deg \bar g_1$ and $\deg g_2=\deg \bar g_2$ and by assumption $(g_1,g_2)=1$.  Thus, $g_1=\bar g_1$ and $g_2 =\bar g_2$ forcing every coefficient $a$ of $g_1$ and $g_2$ to satisfy $a^{p^s}=a$.
\end{proof}
We now use Lemma \ref{rational} to prove an analogous results for algebraic extensions of the rational function field.
\begin{lemma}
\label{algebraic}
Let $g, t \in M \setminus \G_p$, let $M_G$ be the Galois closure of $M$ over $\G_p(t)$, let 
\[
[M_G:\G_p(t)]=m,
\]
 and assume that for some positive integer $s$ we have that the number of distinct primes dividing the pole divisor of $\displaystyle\frac{g^{p^s}-g}{t^{p^s}-t}$ is less than a positive constant $C$ while $mh_{M}(g) < p^s -C$.  In this case $g$ satisfies a minimal polynomial with all the coefficients  in $\F_{p^s}(t)$.
\end{lemma}
\begin{proof}
Let $\calV_{M_G}$ be the set of all primes which can appear as poles of $ \displaystyle \frac{g^{p^s}-g}{t^{p^s}-t}$ in $M_G$ and let $\calV_M$ be the set of primes below $\calV_{M_G}$ in $M$.  Note that $|\calV_M| \leq C$ by assumption, and  we have that $g^{p^s}-g \equiv 0 \mod t^{p^s}-t$ in $O_{M_G,\calV_{M_G}}$.  Further for any conjugate $\hat g$ of $g$ over $\G_p(t)$ we also have $\hat g^{p^s}-g \equiv 0 \mod t^{p^s}-t$ in $O_{M_G,\calV_{M_G}}$.  Thus, any elementary symmetric function $A$ of the conjugates will satisfy the same equivalence: $A^{p^s}-A\equiv 0 \mod t^{p^s}-t$ in $O_{M_G,\calV_{M_G}} \cap \G_p(t)=O_{\G_p(t), \calV_{\G_p(t)}}$.  Observe that $\calV_{\G_p(t)}$ contains $C$ primes as they all lie below primes of $\calV_M$, and $h_{\G_p(t)}(A) \leq h_{M}(A) \leq mh_M(g)$.  Thus, $\displaystyle\frac{A^{p^s}-A}{t^{p^s}-t}$ has poles at less than $C$ primes while $h_{\G_p(t)}(A)< p^s-C$.  Now we can apply Lemma \ref{rational} to conclude that $A \in \G_{p^s}(t)$.
\end{proof}

\section{Defining Polynomial Rings}
\setcounter{equation}{0}
Before describing assorted definitions of various polynomial rings, we remark below on the fact that a polynomial ring is a ring of $\calS$-integers of a certain kind and every ring of $\calS$-integers is an integral closure of a polynomial ring.
\subsection{Polynomial Rings as Rings of $\calS_K$ integers}
Let $K$ be a rational function field, let $\qq_K$ be a prime of $K$ of degree one.  In this case $O_{K,\{\qq_K\}}$ is a polynomial ring.  Indeed, let $\pp_K$ be another degree one prime.  Since in any rational field any zero degree divisor  is principal, there exist an element $t \in K$ with a divisor equal to $\displaystyle \frac{\pp_K}{\qq_K}$.  Now $t \in O_{K,\{\qq_K\}}$.  Let $f \in O_{K,\{\qq_K\}}$ be of height $1$.  Since $\pp_K$ is of degree 1, for some constant $a$ we have that $f \equiv a \mod \pp_K$.  Now consider $\displaystyle \frac{f-a}{t} \in O_{K,\{\qq_K\}}$ and note that the height of this element is 0.  ($f-a$ is of height 1 and when we divide by $t$ we eliminate the only zero and the only pole of $f-a$.)  Hence, $f-a=bt$ for some constant $b$.  Now proceeding by induction on the height of $f$, assume the height is $m>0$ and $f\equiv a \mod \pp_K$.  In this case  $\displaystyle\frac{f-a}{t} \in O_{K,\{\qq_K\}}$ and has one less zero than $f-a$.  Assuming inductively that $\displaystyle\frac{f-a}{t}=\sum_{i=0}^{m-1} a_it^i$, we conclude that $f$ is also a polynomial in $t$.

Now let $K$ be a function field and let $\calS_K$ be a non-empty set of its primes.  Let $t$ be any element of the field having poles at elements of $\calS_K$ only and consider the rational subfield of $K$ formed by adjoining $t$ to the constant field of $K$.  It is now not hard to show that the integral closure of the polynomial ring generated by $t$ in $K$ is exactly $O_{K,\calS_K}$.\\

We now proceed to definitions of polynomial rings.
\subsection{Combining the $p$-th power equations and the weak vertical method to define polynomial rings using two universal quantifiers: the uniform version.}
\label{sec:uniform definition}
Given a non-constant element $t \in K$, our goal is to define $\G_p[t]$  using $p$-th power equations and the weak vertical method.  We describe a plan for defining a polynomial ring in a given non-constant field element $t$ using four additional parameters: $a_{\pp_K},b_{\pp_K}, d,g$ with values to be specified below. Here we assume that $t$ is not a $p$-th power in $K$ so that $K/\G_p(t)$ is separable.   We also assume that the constant field has an extension of degree $q \not =p$ and that $\xi_q \in K$.   Thus our definition is uniform across all the fields which have an extension of degree $q$ not equal to the characteristic and has a $q$-th root of unity.  We will be able to treat the  case when only extension of degree $p$ is available but unfortunately the definition there will not be uniform.  It will ultimately depend on the genus of the field.  Below we set up our assumptions and notation.
\begin{notationassumptions}
\label{not:uniform}
\be
\item Let $M$ be the Galois closure of $K$ over $\G_p(t)$.  Let $\alpha \in M$ be a generator of $M$ over $\G_p(t)$.
\item Let $K_1,\ldots, K_r$ be all the cyclic subextensions of $M$ containing $\G_p(t)$ and let $m=[M:\G(t)]$.
\item  Let $\EE_{M}$ be the $M$-divisor obtained by multiplying all the primes of $M$ ramifying in the extension $M/\G_p(t)$.  Let $e_M =\deg \EE_M$.
\item Choose a basis $\Omega_\ell=\{1,\ldots,\omega_{\ell,m_{\ell}}\}$ for $M$ over $K_\ell$.
\item If the constant field of $M$ is infinite, then let $z \in \Z_{>0}$ be such that there is no constant field extension in the extension $\F_{p^z}(t,\alpha, \Omega_1,\ldots, \Omega_r)/\F_{p^z}(t)$.  If the constant field of $M$ is finite, then let $p^z$ be the size of this field.
\item \label{it:d} Choose a prime $\pp_K$ not ramified in the extension $M/K$ and not occurring in the divisor of $t$, and an element $d \in K$ such that for some $u \in \Z_{>0}$ we have that $d=t^{p^{zu}}$ and  $p^{zu} >qm^2\deg \pp_K$.  (We will need to raise other variables to $p^{zu}$.  To do this we need the parameter $d$.  We set $t^{p^e}=d$ and then raise other variables to the power $p^e$ using SPPE equations.)

\item 
Now let be $a_{\pp_K}, b_{\pp_K}$ as in Proposition \ref{finitely many for p not q}. More specifically, we assume $\ord_{\pp_K}a_{\pp_K} =-1$, $a_{\pp_K}$ has only one other pole $\qq_K$ not ramified in the  extension $M/\G_p(t)$,  $\ord_{\qq_K}(b_{\pp_K}-1) >0 $, $b_{\pp_K}$ is not a $q$-th power in the residue field of $\pp_K$, and has a zero at every prime lying above any prime ramifying in the  extension $M/\G_p(t)$ and not a pole of $t$.  Additionally assume 
\[
h_K(b_{\pp_K}) > \max_i\{C_M(\Omega_i) + h_M(\det(\sigma_j(\omega_{\ell,i}))^2\max_{\ell}h_M(\omega_{\ell,i})\} + e_M + 2h_M(t).  
\]
A $b_{\pp_K} \in K$ satisfying Proposition \ref{finitely many for p not q} and the height inequality exists by the  Strong Approximation Theorem.  
\ee
\end{notationassumptions}
 Let $g\in K$ be given.
Now consider the following set of equations holding for some $s,e, w,y, s_1, w_1  \in\Z_{> 0} $.   We divide the equations into two groups with the first group directly below.  The range of values for variables in these equations is either $K$ or   $O_{K,\pp_K}$, as specified.  
\begin{equation}
\label{rangeO}
b_1, b_2,  t_1,t_2, f_1, f_2, A, B,  \in O_{K,\pp_K}, 
\end{equation}
\begin{equation}
\label{rangeK}
f, Y, v, x_0, \ldots, x_{q-1} \in K,
\end{equation}

\begin{equation}
\label{eq:1.0}
b_{\pp_K}=\frac{b_1}{b_2},  
\end{equation}
\begin{equation}
\label{eq:t}
t^{p^{y_1}(p^{{w_1}-1})}=\frac{t_1}{t_2}, At_1 +Bt_2=1
\end{equation}
\begin{equation}
\label{eq:pthpower1}
v =t^{p^{s}}-t,
\end{equation}

\begin{equation}
\label{eq:all poles}
f^{p^y(p^w-1)}=\frac{f_1}{f_2}, \ord_{\pp_K}f \geq 0, f_2 {\large |}_{O_{K,\pp_K}} t_2
\end{equation}
\begin{equation}
\label{eq:2}
\ord_{\pp_K}a_{\pp_K}^{p^s}\leq \ord_{\pp_K}(f_1^{p^{zu}}f_2^{p^{zu}})  +\ord_{\pp_K}(b_1^{p^{zu}}b_2^{p^{zu}}) \land s \equiv 0 \mod z
\end{equation}
\begin{equation}
\label{eq:powerf1}
 Y=f^{qp^{s}}-f^q,
 \end{equation}
\begin{equation}
\label{eq:x1}
 x_i=t^i\frac{Y}{v}, i = 1,\ldots,q-1,
\end{equation}
\begin{equation}
\label{eq:norms1.1}
\exists i =0,\ldots, q-1 \forall c \in K \exists y  \in L_4(\sqrt[q]{c}) \in {\mathbf N}_{L_4(\sqrt[q]{c})/L_4}(y) = x_i,
\end{equation}
where for each $i$ the field $L_4=K(\sqrt[q]{1+x_i^{-1}}, \sqrt[q]{1+(c+c^{-1})x_i^{-1}})$.

\begin{claim}
Equations \eqref{eq:1.0}--\eqref{eq:norms1.1} are satisfied if and only if $f^q \in  \G_p[t]$.  
\end{claim}
\begin{proof}
First assume equations \eqref{eq:1.0}--\eqref{eq:norms1.1} are satisfied.  Now from \eqref{eq:all poles} we know that $f$ has poles only at poles of $t$.  Now from \eqref{eq:2}, using assumptions on $p^{zu}$, we conclude that
\begin{equation}
\label{eq:cube}
p^s\deg \pp_K \geq h_K(f^{p^{zu}}) +h_K(b_{\pp_K}^{p^{zu}}) > qm\deg \pp_Kh_K(f) + qm\deg \pp_Kh_K(b_{\pp_K}).
\end{equation}
Hence,
\begin{equation}
p^s > qm^2h_K(f) + qm^2h_K(b_{\pp_K})  \geq h_M(f^q) + h_K(b_{\pp_K}).
\end{equation}
Now using assumptions on $b_{\pp_K}$ we obtain
\begin{equation}
\label{ineq:height}
p^s  \geq mh_M(f^q) +\max_i\{C_M(\Omega_i) + h_K(\det(\sigma_j(\omega_{\ell,i}))^2\max_{\ell}h_M(\omega_{\ell,i})\} + e_M+ 2h_M(t).  
\end{equation}
Below we consider separately the cases of a finite constant field and an infinite constant field.  
We will address the case of the finite constant field first.\\

{\it The case of a finite constant field:}  Observe that $f \in M \subset \F_{p^s}M$.  From \eqref{eq:norms1.1} we conclude by  Proposition \ref{prop:norm2} and Lemma \ref{le:ordminus1} that for any prime $\aaa_K$ such that $\aaa_K$ is not a pole of $t$, is not ramified in the extension $K/\F_{p^z}(t)$, it is the case that 
\[
 \ord_{\aaa_K}\frac{f^{qp^{s}}-f^q}{t^{p^{s}}- t} >0.
 \]
   Therefore, over $M$ we have that for any prime $\aaa_M$ such that $\aaa_M$ is not a pole or zero of $t$, is not ramified in the extension $\F_{p^s}M/\F_{p^s}(t)$, we have that $\displaystyle\ord_{\aaa_M} \frac{f^{qp^{s}}-f^q}{t^{p^{s}}- t}>0$. 

Let $K_i$ be a cyclic subextension of $M$ containing $\G_p(t)$  as described above, and apply Lemma \ref{use order} to conclude that $f^q \in \F_{p^s}K_i$.  Since $\G_p \subseteq \F_{p^s}$ by the second part of \eqref{eq:2}, we have $\F_{p^s}(t) =\bigcap_{i=1}^r \F_{p^s}K_i$, we conclude that $f^q \in \F_{p^s}(t)$.  Since we know from \eqref{eq:all poles} that $f$ has no poles at any prime other than the prime which is the pole of $t$, we conclude that $f^q \in \F_{p^s}[t] \cap K=\F_{p^z}[t]=\G_p[t]$.\\

{\it The case of an infinite constant field with an extension of degree $q$.}  Here we start with Lemma \ref{algebraic} and \eqref{ineq:height} to conclude using Proposition \ref{prop:norm2} and Lemma \ref{le:ordminus1} again  that the monic minimal polynomial of $f$ over $\G_p(t)$ has coefficients in $\F_{p^s}[t]$.  From this point on we proceed as in the finite case. \\

We now show that if $f \in \G_p[t]$, then all the equations can be satisfied.  We start with \eqref{eq:all poles} and \eqref{eq:t} .  Observe that since every divisor of $K$ corresponds to a divisor of some global field and global fields have finite class numbers with respect to zero degree divisors, for every zero-degree divisor of $K$, there exists a power $n$ such that raised to power $n$, the divisor becomes principal.  Now for any positive integer $n$ for some $w, y \in \Z_{\geq 0}$  we have that $p^y(p^w-1)$  is divisible by $n$.  Let the divisor of $f$ be of the form $\displaystyle \frac{\AA_K}{\BB_K}$, where $\AA_K$ and $\BB_K$ are integral divisors  of $K$ without common factors. Note that since $f \in \G_p[t]$, all the factors of $\BB_K$ must occur in the pole divisor of $t$ in $K$.  Further the divisor $\displaystyle \frac{\BB_K^{\deg \pp_K}}{\pp_K^{\deg \BB}}$ is a  zero degree divisors and so for some $n \in \Z_{>0}$ there exists an element $x \in K$ such that the divisor of $x$ is $\displaystyle \left (\frac{\BB_K^{\deg \pp_K}}{\pp_K^{\deg \BB}}\right )^n$.  Thus 
\[
f^{n\deg \pp_K}=\frac{\hat x}{x}, 
\]
where zeros of $x$ are exactly the poles of $f$ and $x, \hat x \in O_{K, \{\pp_K\}}$.  As mentioned above for some positive integers $y,w, \ell$ we have that $p^y(p^w-1)=\ell n\deg \pp_K  $, so that 
\[
f^{p^y(p^w-1)}=\left (\frac{\hat x}{x} \right )^{\ell}
\]
	and we can set $f_2=x^{\ell}$.  Clearly the divisibility condition in \eqref{eq:all poles} can now be satisfied over $O_{K, \{\pp_K\}}$ as the zero divisor of $t_2$ should have occurrences of factors of $\BB_K$ to sufficiently large power if we select $y_1, w_1$ to be large enough.  A similar argument and the Strong Approximation Theorem show that \eqref{eq:t} can be satisfied also.
	
	  The only other equation we need to consider is \eqref{eq:norms1.1}.  (All other equations can clearly be satisfied.)  Now, in order for \eqref{eq:norms1.1} to hold for some $i=0,\ldots,q-1$, and some $s$, we just need to make sure that for some $i=0,\ldots, q-1$ all the poles of 
\begin{equation}
\label{eq:ratio} 
t^i\frac{f^{qp^{s}}-f^q}{t^{p^{s}}-t} 
\end{equation}
are of order divisible by $q$.  However, if $f$ is a polynomial in $t$ over $\F_{p^{s}}$, the ratio in \eqref{eq:ratio} will have a pole only at the pole of $t$.  Since for some $i$ the degree of the polynomial equal to the ratio must be divisible by $q$, we have that \eqref{eq:norms1.1} can be satisfied by Proposition \ref{prop:norm2}.\\

The only remaining task is  to make sure that a given function, not just its $q$-th power, is in $\G_p[t]$. To this end we add some equations identical (except for the obvious changes in the names of some variables) to \eqref{rangeO}-\eqref{eq:norms1.1} saying  for some $E \in K$ with its  zero divisor divisible by all primes ramified in $M/\G_p(t)$, that  $E^q \in \G_p[t]$.    (Note that this divisibility condition can be arranged over $O_{K,\pp_K}$ by requiring that the ``numerator'' of some power of $E$  with respect to the ring is divisible by $b_1$, where $b_1$ is defined in \eqref{eq:1.0}.)  Further, we add the equation

\begin{equation}
\label{eq:f0}
f=E^qg+1
\end{equation}
and an equation saying that $g$ has poles at poles of $t$ only (similarly, by forcing the ``denominator'' of some power of $g$ to divide the denominator of some power of $t$ in $O_{K,\pp_K}$.)  This is the second group of equations we mentioned earlier.

   Suppose now that $f \not \in \G_p(t)$ while $f^q \in \G_p[t]$.  In  this case we conclude that $f^q$ is not a $q$-th power in $\G_p(t)$.  Further, since the constant field of $K$ is the same as the constant field of $\G_p(t)$, we conclude that at least one zero of $f^q$ must be of order not divisible by $q$.  In this case, the prime corresponding to this zero is ramified in the extension $K/\G_p(t)$.  However, we know that $E$ has a zero at every prime ramifying in the extension $K/\G_p(t)$ and not a pole of $t$, and we also know that $g$ has a pole at poles of $t$ only.  Thus $f$ does not have a zero at any prime ramifying in the extension $K/\G_p(t)$.  Consequently, $f \in \G_p[t]$ and $g \in \G_p(t)$ with all of its poles at the primes that are poles of $t$.  Hence $g \in \G_p[t]$.

\end{proof}
We now have all the pieces together for the following theorem.
\begin{theorem}
\label{uniform poly}
If $K$ is a function field in one variable over a field of constants $\G_p$ algebraic over a finite field of $p$ elements,  $q$ is a prime number different from $p$, $t$ is a non-constant element of  $K$ and the following conditions are satisfied:
\be
\item $\G_p$ has an extension of degree $q$,
\item $\G_p$ contains a primitive $q$-th root of unity,
\item $t$ is not a $p$-th power (or, alternatively $K/\G_p(t)$, is separable),
\ee
then the polynomial ring $\G_p[t]$ has a $\exists \ldots \exists \forall\forall \exists \ldots \exists$-definition over $K$ requiring three additional parameters (besides $t$), and for a fixed $q$ this definition is uniform across all the fields and elements $t$ satisfying the conditions above.
\end{theorem}
\begin{proof}
Choose a prime $\pp_K$ and set the values of parameters $a_{\pp_K},b_{\pp_K}, d$ as in Section \ref{sec:uniform definition}.    We can now rewrite all the equations above  so that all the variables range over $K$ and convert $p$-th power equations to uniform polynomial ones.  We can do this using the fact that we have a definition of $O_{K,\pp_K}$ and over $O_{K,\pp_K}$ the $p$-th power equations and ``the same power equations'' are uniform. We can also use Lemma \ref{le:divbyz} to rewrite congruences on exponents (e.g. the equivalence ``$s \equiv 0 \mod z$'' in  \eqref{eq:2}). Further, we can re-use one of the variable in the range of the universal quantifiers we deployed to define $O_{K,\pp_K}$ in the norm equations  so that we do not need any additional universal quantifiers beyond the two needed to define $O_{K, \pp_K}$ uniformly.
\end{proof}
As we have discussed in the introduction, we can specialize the theorem above to the class of global fields to get the following corollary.
\begin{corollary}
If $K$ is a global field of characteristic $p$ greater than 2 or of characteristic equal to 2 with the constant field containing subfield of size 4, and a non-constant $t \in K$ is not a $p$-th power of $K$, then the polynomial ring of $t$ over the largest  finite field contained in $K$ has a uniform (in $K$ and $p$) definition over $K$ of the form $\exists \ldots \exists \forall\forall \exists \ldots \exists$ using three additional parameters (besides $t$).
\end{corollary}

One can also define a polynomial ring over a finite field of a specific size or its intersection with the constant field of $K$.
\begin{corollary}
\label{cor:specific size}
\label{uniform poly}
If $K$ is a function field in one variable over $\G_p$,  $q$ is a prime number different from $p$, $z$ is a positive integer, $t$ is a non-constant element of  $K$ and the following conditions are  satisfied:
\be
\item $ \G_p$ has an extension of degree $q$,
\item $\G_p$ contains a primitive $q$-th root of unity,
\item $t$ is not a $p$-th power (or, alternatively $K/\hat \G_p(t)$, is separable),
\ee
then the polynomial ring $\G_{p^z}[t] \cap K$ has a $\exists \ldots \exists \forall\forall \exists \ldots \exists$-definition over $K$ requiring three additional  parameters (besides $t$), and this definition is uniform across all the fields and elements $t$ satisfying the conditions above.
\end{corollary}
\begin{proof}
Suppose we want to define $\G_p[t] \cap \F_{p^z}[t]$ for some fixed $z$.  It is enough to take an element $f \in \G_p[t]$ and  some positive integer $s$ such that $f | (t^{p^s}-t)$ in the polynomial ring and then require two divisibility conditions: $f^{p^s}-f$ is divisible by $t^{p^s}-t$ and $f^{p^{z+s}}-f$ is divisible by $t^{p^{z+s}}-t$ in the polynomial ring.
\end{proof}
\begin{remark}
\label{rem:pthpower}
If $t=w^{p^m}$ and $w$ is not a $p$-th power, then defining the polynomial ring in $t$ is not complicated.  All one needs to do is to define the polynomial ring in $w$ and then define (in a uniform fashion) the set of $p^s$ powers of the elements of the polynomial ring where $s \equiv 0 \mod m$.  Thus we can remove a restriction on $t$ by adding another parameter or more universal quantifiers (to state that some variable is not a $p$-th power).
\end{remark}
We now turn our attention to non-uniform definitions.  This will allow us to use $p$-th power equations defined over the field itself, without a detour to a ring of $\calS_K$-integers.  
We start with the function field with a field of constants having an extension of degree prime to characteristic.
\section{One quantifier definitions of polynomials over  function fields with a constant field extension of degree $q$}

\setcounter{equation}{0}
\begin{notation}
We now review and  specialize our assumptions from  preceding sections.
\label{not:2}
\begin{itemize}
\item Let $\G_p$ have an extension of degree $q$ and let $\xi_q \in \G_p$.
\item Let $t \in K \setminus \G_{p}$ be such that $t$ is not a $p$-th power in $K$ or in other words the extension $K/\G_p(t)$ is separable.
\item Let $\pp_{K,\infty}$ be a pole of $t$.
\item Let $e=e(\pp_{K,\infty}/\pp_{G(t),\infty})$ be the ramification degree of $\pp_{K,\infty}$ over $\G_p(t)$.
\item Let $o$ be the smallest positive integer such that for some $K$-prime $\aaa_K$ we have that $\ord_{\aaa_K}t=o$.  
\end{itemize}
\end{notation}

\begin{lemma}
\label{le:bigc}
If $\pp_K$ is a prime of $K$  such that $\ord_{\pp_K}t \geq 0$, then for any sufficiently large  $n \in \Z_{>0}$  there exists an element $b \in K$ such that $b$ is not a $q$-th power modulo $\pp_K$ and $\ord_{\pp_{K,\infty}}b =-ne$.
\end{lemma}
\begin{proof}
The residue field of $\pp_K$ is a finite extension of $\G_p$, it contains $\xi_q$, has an extension of degree $q$ which must be cyclic, and therefore  it contains elements that are not $q$-th powers.   Let $b \in K$ be such that that its residue class mod $\pp_{K}$ is not a $q$-th power in the residue field. (This assumption implies that $\ord_{\pp_K}b=0$.)   Let $\pp_{\G_p(t)}$ be the prime of $\G_p(t)$ lying below $\pp_{K}$ and let $P(t)\in \G_p[t]$ be the monic irreducible polynomial corresponding to $\pp_{\G_p(t)}$.  Let $d=\ord_{\pp_{K,\infty}}b$ and observe that $b + t^{m+d}P(t) \equiv b \mod \pp_K$ while
\[
\ord_{\pp_{K,\infty}}(b + t^{m+d}P(t))=-e(m+d+\deg P(t)).
\]
Consequently, as long as $n \geq \deg P(t)+d$, we can find an element of $K$ such that it is equivalent to $b$ mod $\pp_K$ and has order exactly $-ne$ at $\pp_{K,\infty}$.
\end{proof}

We are now ready for the main result of this section.
\begin{theorem}[Defining polynomials using one quantifier]
\label{thm:best}
$\G_p[t]$ is definable over $K$ using one universal quantifier.
\end{theorem}
\begin{proof}
We assume that $t$ is not a $p$-th power in $K$.  Otherwise as we discussed before in Remark \ref{rem:pthpower}, we replace $t$ by a parameter $w$ such that $t=w^{p^s}$ and $w$ is not a $p$-th power. The polynomial ring of $t$ is then existentially definable in the polynomial ring of $w$.   Let $E(t)$ be the polynomial divisible by all primes which ramify in $K/\G_p(t)$ and are not poles of $t$.  Fix $f \in K$ and consider the following first-order statement $\tt S$ in the language of  rings:
\[
\forall c \in K \exists v, \hat v, \tilde t, \tilde v \in K
\]

\begin{equation}
\label{clause :1}
(\exists s \in \Z_{>0}: f^{p^s}=f) \lor (\ord_{\pp_{K,\infty}} c > e \ord_{\pp_{K,\infty}}t) \lor [ (\exists  s \in \Z_{\geq 0}: v=t^{p^s}) \land (e \ord_{\pp_{K,\infty}}v^p < \ord_{\pp_K} c  <e \ord_{\pp_{K,\infty}}v)]
\end{equation}
or
\begin{equation}
\label{eq:clause2}
\exists  s \in \Z_{\geq 0}:[ v=t^{p^s} \land \tilde v=\tilde t ^{p^ s}]
\end{equation}
and
\begin{equation}
\label{eq:orderpole}
\ord_{p_{K,\infty}}c  = e \ord_{\pp_{K,\infty}}v,
\end{equation}
and
\begin{equation}
\label{hat power}
\exists  \hat s \in \Z_{\geq 0}: \hat v=t^{p^{\hat s}}, 
\end{equation}
and
\begin{equation}
\label{compare powers}
\exists \tilde s:  \frac{\tilde v^o}{\tilde t^o}=\left (\frac{\hat v}{t}\right)^{p^{\tilde s}}
\end{equation}
and
\begin{equation}
\label{eq:orderzero}
 \ord_{\pp_K}f \geq 0 \mbox{ for all $K$-primes } \pp_K \mbox{ ramifying in the extension } K/\G_p(t) \mbox{ and not poles of t}
\end{equation}
and
\begin{equation}
\label{eq:zerooft}
\ord_{\pp_K}f \geq 0 \mbox{ for all } \pp_K \mbox{ such that } \ord_{\pp_K}t >0
\end{equation}
and
\begin{equation}
\label{eq:qth1}
\bar f=E(t)f+1, 
\end{equation}
and
\begin{equation}
\label{eq:x1nonu}
 x_{i,1}=t^i\frac{f^{qp^{\hat s}}-f^q}{t^{p^{\hat s}}-t},  i = 0,\ldots,q-1,
\end{equation}
and
\begin{equation}
\label{eq:x2}
x_{i,2} =t^i\frac{\bar f^{qp^{\hat s}}-\bar f^q}{t^{p^{\hat s}}-t},  i = 0,\ldots,q-1,
\end{equation}
and
\[
\mbox{for all } j=1,2, i =0,\ldots,q-1,
\]
\begin{equation}
\label{eq:orderzerox}
 \ord_{\pp_K}x_{i,j} \geq 0 \mbox{ for all $K$-primes } \pp_K \mbox{ ramifying in the extension } K/\G_p(t) \mbox{ and not poles of }t,
\end{equation}

\begin{equation}
\label{eq:normsi1}
\exists i  \in \{0,\ldots,q-1\}  \exists y \in L_4(\sqrt[q]{c}) : {\mathbf N}_{L_4(\sqrt[q]{c})/L_4}(y) = x_{i,1},
\end{equation}
and
\begin{equation}
\label{eq:normsi2}
\exists i  \in \{0,\ldots,q-1\} \exists \tilde y \in L_4(\sqrt[q]{c}) :   {\mathbf N}_{L_4(\sqrt[q]{c})/L_4}(\tilde y) = x_{i,2},
\end{equation}
where where for each $i$ the field $L_4=K(\sqrt[q]{1+x_i^{-1}}, \sqrt[q]{1+(c+c^{-1})x_i^{-1}})$.

Suppose {\tt S} is true, $f$ is not a constant and for some $c$ we have that equations \eqref{eq:clause2} and \eqref{eq:orderpole} hold. (Note that in this case $c \not =0$.)  In this case, equation \eqref{compare powers} also holds, i.e. 
\begin{equation}
\tilde t^{o(p^s-1)} =t^{p^{\tilde s}(p^{\hat s}-1)}
\end{equation}
If $\aaa_K$ is a prime of $K$ such that $\ord_{\aaa_K}t=o$, then $\aaa_K$ must be a zero of $\tilde t$ and we have 
\[
(\ord_{\aaa_K}\tilde t)o(p^s-1)=op^{\tilde s}(p^{\hat s}-1).
\]
  Consequently, we must have $(p^s-1) | (p^{\hat s}-1)$ and $s |\hat s$.

Now assume that {\tt S} is true while for some prime $\qq_K$ such that $\ord_{\qq_K}t = 0$ we have that $\ord_{\qq_K}f <0$.  (Note that by \eqref{eq:orderzero} and \eqref{eq:zerooft} we cannot have the case where $\ord_{\qq_K}f<0$ while $\ord_{\qq_K}t>0$ or $\qq_K$ ramifies in the extension $K/\G_p(t)$.)  Let $\bar s$ be the smallest positive integer $s$ such that $\ord_{\qq_K}(t^{p^s}-t)>0$.  Such an $\bar s$ exists by Proposition \ref{prop:pthpower3}. Let $c \in K$ be such that $c$ is not a $q$-th power modulo $\qq_K$ and $\ord_{\pp_{K,\infty}}c=ep^s$ with $s \equiv 0 \mod \bar s$.  Existence of such a $c$ follows by Lemma \ref{le:bigc}.  For this $c$ equations \eqref{eq:clause2}-- \eqref{eq:normsi2} must hold. 

As discussed above, from \eqref{compare powers} we also have that if $s \equiv 0 \mod \bar s$, then $\hat s \equiv 0 \mod \bar s$ too.   We now deduce that  $\ord_{\qq_K}x_{i,1} <0$ for all $i=0,\ldots,q-1$ and $\ord_{\qq_K}x_{i,1} \not \equiv 0 \mod q$ for all $i=0,\ldots,q-1$ by Proposition \ref{prop:pthpower3}.  Thus, by Proposition \ref{prop:norm2} we conclude that none of the norm equations in \eqref{eq:normsi1} can hold and therefore we obtain a contradiction.  Consequently, if for some $f$ the sentence above is true, we must conclude that all the poles of $f$ are factors of the pole divisor of $t$.  

Further, for $s$ sufficiently large and such that $\F_{p^s}(t)$ contains all the coefficients of the monic irreducible polynomial of $f^q$ over $\G_p(t)$, we can apply Proposition \ref{use order} (the weak vertical method) to conclude that $f^q \in \G_p(t)$.  Hence $f^q$ is a polynomial.  Similarly, we conclude that $\bar f^q \in \F_p[t]$.  Now we use the same argument as in Section \ref{sec:uniform definition} to conclude that $f \in \F_p[t]$.

We will now assume that $f$ is a polynomial over $\G_p$ and show that the sentence {\tt S} above is true for  $f$.  Without loss of generality we can assume that $f$ is not a constant and so $f^{p^s} \not =f$ for any positive integer $s$.  Let $z$ be a positive integer such that the coefficients of $f$ and $\bar f$ are in $\F_{p^z}$.  Let  $s_1, s_2$ be positive integers such that $o |p^{s_1}(p^{s_2}-1)$. Let $c \in K$ be given. If 
\[
\forall s \in \Z_{\geq 0}: \ord_{\pp_{K,\infty}}c \not = -ep^s,
\]
 then \eqref{clause :1}  is true and {\tt S} is true.  So we can assume that 
\[
 \exists s \in \Z_{\geq 0}:\ord_{\pp_{K,\infty}}c=-ep^s 
\]
    Let $\hat s$ be a multiple of $szs_2$ and let $\tilde s=s_1$ so that
\[
o(p^s-1)|p^{\tilde s}(p^{\hat s}-1).
\]
Now let 
\[
w=\frac{p^{\tilde s}(p^{\hat s}-1)}{o(p^s-1)}, \tilde t =t^w,
\]
 and observe that \eqref{compare powers} now holds.   The finite field of $p^{\hat s}$ elements contains the coefficients of $f$ and $\bar f$, and therefore by Lemma \ref{will work}, for some values of $i_1, i_2=0,\ldots, q-1$ we have that $x_{i_1,1}$ and $x_{i_2,2}$ are both polynomials of degree divisible by $q$.  Therefore, by Proposition \ref{prop:norm2}, we can satisfy  \eqref{eq:normsi1} and \eqref{eq:normsi2}.   We should finish, as before, by noting that the norm equations can be rewritten in the polynomial form with all the variables ranging over $K$.

\end{proof}

\section{Defining polynomial using extensions of degree $p$ in characteristic $p$}
We now adjust the discussion from the preceding section to the situation when the constant filed has extensions only of degree divisible by the characteristic.  First we need a version of Lemma \ref{le:bigc} proved in the same fashion using Lemma \ref{le:finite}.
\begin{lemma}
\label{le:bigcp}
If $\pp_K$ is a prime of $K$  such that $\ord_{\pp_K}t \geq 0$, then for any sufficiently large  $n \in \Z_{>0}$  there exists an element $b \in K$ such that the equation $T^p-b^{p-1}T -1=0$ has no solution modulo $\pp_K$ and $\ord_{\pp_{K,\infty}}b =-ne$, where $e$ is the ramification degree of $\pp_{K,\infty}$ over $\G_p(t)$.
\end{lemma}

We now construct the definition of a polynomial ring for this case.

\begin{proposition}
Let $K$ be a function field over the field of constants $\G_p$ and assume that $\G_p$ has an extension of degree $p$.  In this case for any non-constant $t$, the ring $\G_p[t]$ has a first-order definition over $K$ of the form $\forall \exists \ldots \exists $.
\end{proposition}
\begin{proof}
As before, without loss of generality we assume that $t$ is not a $p$-th power in $K$.   Now fix $f \in K$ and consider the following first-order statement $\tt S_p$ in the language of  rings:
\[
\forall c \in K \exists v, \hat v, \tilde t, \tilde v \in K
\]

\begin{equation}
\label{clause :1p}
(\exists s \in \Z_{>0}: f^{p^s}=f) \lor (\ord_{\pp_{K,\infty}} c > e \ord_{\pp_{K,\infty}}t) 
\end{equation}
\[
\lor [ (\exists  s \in \Z_{\geq 0}: v=t^{p^s}) \land (e \ord_{\pp_{K,\infty}}v^p < \ord_{\pp_K} c  <e \ord_{\pp_{K,\infty}}v)]
\]
or
\begin{equation}
\label{eq:clause2p}
\exists  s \in \Z_{\geq 0}:[ v=t^{p^s} \land \tilde v=\tilde t ^{p^ s}]
\end{equation}
and
\begin{equation}
\label{eq:orderpolep}
\ord_{p_{K,\infty}}c  = e \ord_{\pp_{K,\infty}}v,
\end{equation}
and
\begin{equation}
\label{hat powerp}
\exists  \hat s \in \Z_{\geq 0}: \hat v=t^{p^{\hat s}}, 
\end{equation}
and
\begin{equation}
\label{compare powersp}
\exists \tilde s:  \frac{\tilde v^o}{\tilde t^o}=\left (\frac{\hat v}{t}\right)^{p^{\tilde s}}
\end{equation}
and
\begin{equation}
\label{eq:orderzerop}
 \ord_{\pp_K}f \geq 0 \mbox{ for all $K$-primes } \pp_K \mbox{ ramifying in the extension } K/\G_p(t) \mbox{ and not poles of t}
\end{equation}

and
\begin{equation}
\label{eq:xp}
 x=\frac{f^{p^s}-f}{t^{p^s}-t},
\end{equation}
and

\begin{equation}
\label{eq:orderzeroxp}
 \ord_{\pp_K}x \geq 0 \mbox{ for all $K$-primes } \pp_K \mbox{ ramifying in the extension } K/\F_p(t) \mbox{ and not poles of }t,
\end{equation}
and
\begin{equation}
\label{pext}
\gamma^p-c^{p-1}\gamma-1=0,  
\end{equation}
and
\begin{equation}
\label{eq:normsi1p}
\exists r \in \Z_{>0}:  (1 +c^{p^r}x \not =0) \land {\mathbf N}_{G(\gamma)/G}(y) =1+ c^{p^r}x,
\end{equation}
where $G=K(\delta)$ and $\delta$ is a root of the polynomial $T^p+T+\frac{1}{1 + c^{p^r}x}$. The fact that  sentence  ${\tt S_p}$  is true  if and only if $f$ is a polynomial in $t$ follows by the same argument as in Proposition \ref{thm:best}, except that we need to use Propositions  \ref{eq:ordp} and \ref{prop:need later} in place of Propositions \ref{prop:norm} and \ref{prop:norm2}.

\end{proof}
\section{Higher Transcendence Degree}
\setcounter{equation}{0}
We will now allow the constant field to be arbitrary while not containing the algebraic closure of a finite field and add to our notation list.  As above we also consider separately the case where we use extensions of degree of $q\not=p$ and extensions of degree $p$.  We start with describing notation and assumptions common to both cases.
\begin{notationassumptions}
\label{not:transcendental}
\begin{itemize}
\item  Let $H$ be an arbitrary extension of $\G_p$ such that $\G_p$ is algebraically closed in $H$.
\item Let $t$ be transcendental over $H$.
 \item Let $M$ be a finite Galois extension  of $H(t)$ of degree $d_M$.
\item Let $K$ be the algebraic closure of $\G_p(t)$ in $M$.

\item Let $o$ be the smallest positive integer such that for some $M$-prime $\aaa_M$ we have that $\ord_{\aaa_M}t=o$, let $s_1, s_2$ be such that $o|p^{s_1}(p^{s_2}-1)$ and let $s_0=\ord_qs_2$.  
\end{itemize}
\end{notationassumptions}
We now separate the case where we use extensions of degree $q \not= p$ and the case where we use extensions of degree $p$.
\subsection{Using extensions of degree $q$}
\begin{notationassumptions}
\begin{itemize}
\item Assume $\G_p$ has an extension of degree $q$ and a primitive $q$-th root of unity.

\item Let $\G_p^{\infty}$ be a possibly infinite algebraic extension  of $\G_p$ obtained by adjoining all the elements of $\tilde \F_p$ whose degree over $\G_p$ is prime to $q$.  (Observe that $\G_p^{\infty}$ still has an extension of degree $q$ and if $c \in \G_p$ is not a $q$-th power in $\G_p$, then $c$ is not a $q$-th power in $\G_p^{\infty}$.)
\item Let $m_q$ be such that $\G_p$ contains $\F_{q^{m_q}}$ but not $\F_{q^{m_q+1}}$.
\item Let $N=\G_p^{\infty}M$. (Observe that $[N:\G^{\infty}_pH(t)]=d_M$ and the extension $N/\G^{\infty}_pH(t)$ is normal.)
\item Let $n_q=\ord_{q}d_M$. 
\item Let $\gamma$ generate  an extension of $\G^{\infty}_p$ of degree $q^{n_q}$.
\end{itemize}

\end{notationassumptions}
\begin{lemma}
\label{le:alclose}
If $\pp_{N(\gamma)}$ is a prime of $N(\gamma)$ lying above a $\G_p^{\infty}H(\gamma, t)$-prime $t-a$ with $a \in \G_p^{\infty}$, then $\G_p^{\infty}(\gamma)$ is algebraically closed in the residue field of $\pp_{N(\gamma)}$.
\end{lemma}
\begin{proof}
Let $\pp_N$ lie below $\pp_{N(\gamma)}$ in $N$ and note that $t-a$ is also the $\G^{\infty}_pH(t)$ prime below it.  Since the extension is Galois, the relative degree $f$ of $\pp_N$ over $t-a$ divides $d_M$.  As there is no constant field extension, and $t-a$ is of degree 1 in the rational field, we conclude that the degree of $\pp_N$ divides $d_M$.  Thus, if $\hat \G_p^{\infty}$ is the algebraic closure of $\G_p^{\infty}$ in the residue field of $\pp_N$, we must have that $[\hat \G_p^{\infty}:\G_p^{\infty}]$ is a power of $q$ and divides $d_M$.  Hence, $[\hat \G_p^{\infty}:\G_p^{\infty}]=q^{\ell}$ with $\ell$ being a positive integer less or equal to $n_q$.  Observe that $\gamma$ is of degree $q^{n_q-\ell}$ over $\hat \G_p^{\infty}$.  

We adjoin $\gamma$ to $N$ in two steps.  First we adjoin $\hat \G_p^{\infty}$.  Let $\beta \in \tilde \F_p$ be a generator of this extension over $N$ and $\G_p^{\infty}$, and note that modulo $\pp_N$ the monic irreducible polynomial of $\beta$ splits completely and thus the relative degree of any $N(\beta)$-factor $\pp_{N(\beta)}$ of $\pp_N$ is one.  Hence the residue fields of $\pp_{N(\beta)}$ and $\pp_N$ are the same and $\hat \G_p^{\infty}$ is algebraically closed in the residue field of $\pp_{N(\beta)}$. 

 We now adjoin $\gamma$, whose monic irreducible polynomial remains prime modulo $\pp_{N(\beta)}$, and thus $\pp_{N(\beta)}$ does not split in this extension.  The residue field of its single factor $\pp_{N(\gamma)}$ will be obtained by adjoining $\gamma$ to the residue field of $\pp_{N(\beta)}$.  Thus the residue field of $\pp_{N(\gamma)}$ will contain $\G_p^{\infty}(\gamma)$ and  $\G_p^{\infty}(\gamma)$ will be algebraically closed in this residue field. 
\end{proof}
We now have the following corollary.

\begin{corollary}
\label{cor:notqth}
Let $c \in \G_p(\gamma)\setminus (\G_p(\gamma))^q$,  i.e. $c$ is not a $q$-th power in $\G_p(\gamma)$.  In this case the following statements are true: 
\be
\item $c$ is not a $q$-th power in $\G^{\infty}_p(\gamma), M(\gamma), N(\gamma)$
\item $c$ is not a $q$-th power modulo all $N(\gamma)$-primes $\pp_{N(\gamma)}$ lying above an $\G_p^{\infty}H(\gamma, t)$-prime $t-a$, with $a \in \G^{\infty}_p$.
\ee
\end{corollary}
We now add to our notation list.  
\begin{notation}
Let $c \in \G_p(\gamma)$ denote an element that is not a $q$-th power in the field.
\end{notation}
We will need the following general fact.
\begin{lemma}[Lemma 2.3 of \cite{Sh8}]
\label{le:algebraic}
Let $M/F$ be an arbitrary field extension and let $h(z) \in M(z)$ be such that for infinitely many values of $a \in F$ we have that $h(a) \in F$.  In this case $h(z) \in F(z)$.
\end{lemma}
We now apply the lemma above to a specific class of rational functions.
\begin{lemma}
\label{le:manys}
If $f \in H(t) \subset \tilde \F_pH(t)$ and for infinitely many $s$  there exists a distinct $a \in \F_{p^s}$ (a different $a$ for different $s$) such that $\displaystyle\frac{f^{qp^s}-f^q}{t^{p^s}-t}$ has no pole in $\tilde \F_pH(t)$ at $t-a$, then  $f \in \tilde \G_p(t)$. 
\end{lemma}
\begin{proof}
 If  $\displaystyle \frac{f^{qp^s}-f^q}{t^{p^s}-t} \in \tilde \F_pH(t)$ has no pole at the prime corresponding to a polynomial $t-a$, with $a \in \F_{p^s}$, then $f^{qp^s}-f^q$ has a zero at this prime.   Hence for some $b \in \F_{p^s}$ we have that $f^q-b$ has a zero at $t-a$ or $f^q(a)=b$. Thus, for infinitely many $a \in \tilde \F_p$, we have that $f^q(a) \in \tilde \F_p$ and $f(a) \in \tilde \F_p(t)$ Therefore, by Lemma \ref{le:algebraic}, we have that $f \in H(t) \cap\tilde \F_p(t)=\G_p(t)$.
\end{proof}
Our next step is to revisit Proposition \ref{prop:norm2} to adjust it for an arbitrary constant field.  We will do this in two lemmas.  The first lemma is of a general nature.  The proof is left to the reader.
 \begin{lemma}
 \label{le:samenorm}
If $G/R$ is a field extension of fields of characteristic not equal to $q$, $\xi_q \in R$, and $c \in R$ is not a $q$-th power in $G$, then for any element $y \in R(\sqrt[q]{c})$ we have that 
\[
{\mathbf N}_{G(\sqrt[q]{c})/G}(y)={\mathbf N}_{R(\sqrt[q]{c})/R}(y).
\]
\end{lemma}
The next lemma is a variation on the theme of Proposition \ref{prop:norm2}.
\begin{proposition}
\label{prop:norm2t}
If $x \not =0, x \in M$ and there exists $y \in M(\gamma, \sqrt[q]{1+x^{-1}},\sqrt[q]{c}) $ such that
\begin{equation}
\label{eq:norm2t}
{\mathbf N}_{M(\gamma, \sqrt[q]{1+x^{-1}},\sqrt[q]{c})/M(\gamma, \sqrt[q]{1+x^{-1}})}(y) = x,
\end{equation}
then for any  prime $\pp_{M(\gamma)}$ of $M(\gamma)$ it is the case that one of the following conditions hold:
\be
\item \label{nass1:1}$c$ is a $q$-th power mod $\pp_{M(\gamma)}$, or
\item \label{nass1:2}$\ord_{\pp_{M(\gamma)}} x \geq 0$, or
\item \label{nass1:4}$\ord_{\pp_{M(\gamma)}}x \equiv 0 \mod q$.
\ee
At the same time if $x \in \G_p[t]$ and $x$ is of degree divisible by $q$, then \eqref{eq:norm2t} has a solution  $y \in M(\gamma, \sqrt[q]{1+x^{-1}},\sqrt[q]{c}) $. 

\end{proposition}
\begin{proof}
The proof of the lemma, as before, is based on the following ideas:
\be
\item Any prime in the pole divisor of $x$ splits completely in the extension 
\[
M(\gamma, \sqrt[q]{1+x^{-1}})/M(\gamma),
\]
 and in $M(\gamma, \sqrt[q]{1+x^{-1}})$ the zero divisor of $x$ is a $q$-th power of another divisor.  
\item If none of the conditions \eqref{nass1:1}--\eqref{nass1:4} is satisfied, then for all factors $\pp_{M(\gamma,\sqrt[q]{1+x^{-1}})}$ of $\pp_{M(\gamma)}$ we have that $\ord_{\pp_{M(\gamma,\sqrt[q]{1+x^{-1}})}}x \not \equiv 0 \mod q$ and $c$ is not a $q$-th power mod $\pp_{M(\gamma,\sqrt[q]{1+x^{-1}})}$.  
\ee
As before the last item implies \eqref{eq:norm2t} has no solution.  At the same time, if $x \in \G_p[t]$ and $x$ is of degree divisible by $q$, given that $c$ is a constant and therefore no prime is ramified in this extension, the first item assures that 
\begin{equation}
\label{eq:norm3t}
{\mathbf N}_{K(\gamma, \sqrt[q]{1+x^{-1}},\sqrt[q]{c})/K(\gamma, \sqrt[q]{1+x^{-1}})}(y) = x,
\end{equation}
has a solution $y \in K(\gamma, \sqrt[q]{1+x^{-1}},\sqrt[q]{c}).$  However, by Lemma \ref{le:samenorm}, we have that 
\[
{\mathbf N}_{K(\gamma, \sqrt[q]{1+x^{-1}},\sqrt[q]{c})/K(\gamma, \sqrt[q]{1+x^{-1}})}(y) ={\mathbf N}_{M(\gamma, \sqrt[q]{1+x^{-1}},\sqrt[q]{c})/M(\gamma, \sqrt[q]{1+x^{-1}})}(y),
\]
and therefore \eqref{eq:norm2t} has a solution also.
\end{proof}
We now ready to describe the most important part of the definition of $K$ over $M$.
\begin{lemma}
\label{le:inK}
Let $f \in M$, $f \not \in \G_p$,  and suppose there exists an infinite set $S$ of positive integers and a positive integer $s_b$ such that for all $s \in S$ we have that $\ord_qs \leq s_b$ and the following equations hold:
\begin{equation}
\label{eq:xt}
 x_{i,s}=t^i\frac{f^{qp^{ s}}-f^q}{t^{p^{ s}}-t},  i = 0,\ldots,q-1,
\end{equation}
and
\begin{equation}
\label{eq:normst}
\exists i  \in \{0,\ldots,q-1\}  \exists y  \in M(\gamma, \sqrt[q]{1+x_{i,s}^{-1}},\sqrt[q]{c}): {\mathbf N}_{M(\gamma, \sqrt[q]{1+x_{i,s}^{-1}},\sqrt[q]{c})/M(\gamma, \sqrt[q]{1+x_{i,s}^{-1}})}(y) = x_{i,s}.
\end{equation}
In this case $f \in K$.  

At the same time, if $f \in \F_{p^{z}}[t]$, and all $s \in S$, the infinite set above, are divisible by $z$, then  for all $s \in S$ for some $i$ the norm equation  \eqref{eq:normst} has a solution $y \in M(\gamma, \sqrt[q]{1+x_{i,s}^{-1}},\sqrt[q]{c})$.
 
\end{lemma}
\begin{proof}
First of all observe that since $S$ is infinite and $\ord_qs$ is bounded, as $s \rightarrow \infty$ we also have $|\F_{p^s}\cap \G^{\infty}_p| \rightarrow \infty$.   Since $c$ is not $q$-th power in $N(\gamma)$, by Lemma \ref{le:samenorm}, we also have that for all $s \in S$, 
\begin{equation}
\label{eq:normst2}
\exists i  \in \{0,\ldots,q-1\}  \exists y  \in N(\gamma, \sqrt[q]{1+x_{i,s}^{-1}},\sqrt[q]{c}): {\mathbf N}_{N(\gamma, \sqrt[q]{1+x_{i,s}^{-1}},\sqrt[q]{c})/N(\gamma, \sqrt[q]{1+x_{i,s}^{-1}})}(y) = x_{i,s}.
\end{equation}
Therefore, by Proposition \ref{prop:norm2t}, there exists an infinite set $S$ of positive integers  such that for all $s \in S$ we can find a distinct  $a \in \F_{p^s} \cap \G_p^{\infty}$ (a different $a$ for different $s$) such that for all $N(\gamma)$-factors $\pp_{N(\gamma)}$ of $t-a$ we have that 
\[
\ord_{\pp_{N(\gamma)}}\frac{f^{qp^s}-f^q}{t^{p^s}-t} \geq 0 \land \ord_{\pp_{N(\gamma)}}f  \geq 0.
\]
In other words, $f^{qp^s}-f=(t^{p^s}-t)\omega$, where $\omega$ has no poles at any prime lying above the rational prime $t-a$ and $f$ does not have a pole at any prime lying above the rational prime $t-a$.   Let  $A_0 + A_1T + \ldots +T^k$, with $0 \leq k \leq d_M$ be the monic irreducible polynomial of $f^q$ over $\G_p^{\infty}H(t)$.  Since each $A_i \in \G_p^{\infty}H(t)$ is a symmetric function of conjugates of $f^q$ over $\G_p^{\infty}H(t)$, it is not hard to see that each $A_i$ satisfies
\[
A_i^{p^s}-A=(t^{p^s}-t)\hat \omega,
\]
where $\hat \omega \in H(t)$ does not have poles at the prime $t-a$.    Thus by Lemma \ref{le:manys}, $A_i \in \G_p^{\infty}(t)$.  Consequently, $f \in \G_p^{\infty}K$ and since $f \in M$, we have that $f \in K$.

The fact that if $f \in \G_p(t)$, we can satisfy \eqref{eq:xt}  and \eqref{eq:normst} can be proved in exactly same fashion as in Theorem \ref{thm:best}.
\end{proof}
We now ready to produce a first order definition of $K$ over $M$.  Unfortunately it will use many universal quantifiers and the non-uniform version of the $p$-th power equations.
\begin{proposition}
\label{prop:definable}
$K$ is first order definable over $M$.
\end{proposition}
\begin{proof}
Fix $f \in M$ and consider the following first-order statement $\tt T$ in the language of  rings:
\[
\forall v \in M \exists  \hat v, \tilde t, \tilde v \in M
\]

\begin{equation}
\label{clause :1t}
\forall s \in \Z_{\geq 0} \mbox{ with } s \equiv 1 \mod q : v \not =t^{p^s}
\end{equation}
or
\begin{equation}
\label{eq:clause2t}
\exists   s \in \Z_{\geq 0}  \mbox{ with } s \equiv 1 \mod q :[ v=t^{p^s} \land \tilde v=\tilde t ^{p^ s}]
\end{equation}

and
\begin{equation}
\label{hat powert}
\exists i =1,\ldots, q^{s_0+n_q +1}-1 \exists  \hat s \in \Z_{\geq 0}: \hat v=t^{p^{\hat s}}, \hat s \equiv i \mod q^{s_0+n_q+1}
\end{equation}
and
\begin{equation}
\label{compare powerst2}
\exists \tilde s:  \frac{\tilde v^o}{\tilde t^o}=\left (\frac{\hat v}{t}\right)^{p^{\tilde s}}
\end{equation}
and
\begin{equation}
\label{eq:xt2}
 x_i=t^i\frac{f^{qp^{\hat s}}-f^q}{t^{p^{\hat s}}-t},  i = 0,\ldots,q-1,
\end{equation}

and

\begin{equation}
\label{eq:normsit2}
\exists i  \in \{0,\ldots,q-1\}  \exists y \in M(\gamma,\sqrt[q]{1+x^{-1}},\sqrt[q]{c}) : {\mathbf N}_{M(\gamma,\sqrt[q]{1+x_i^{-1}},\sqrt[q]{c})/M(\gamma,\sqrt[q]{1+x_i^{-1}})}(y) = x_{i},
\end{equation}
Suppose {\tt T} is true for some $f \in M$ and we  are given  $v = p^{s}$ with $s \equiv 1 \mod q$.  For such a $v$ the second clause of the statement must be true.  Thus,  equation \eqref{compare powerst2} must hold, i.e. 
\begin{equation}
\tilde t^{o(p^s-1)} =t^{p^{\tilde s}(p^{\hat s}-1)}
\end{equation}
If $\aaa_K$ is a prime of $K$ such that $\ord_{\aaa_K}t=o$, then $\aaa_K$ must be a zero of $\tilde t$ and we have 
\[
(\ord_{\aaa_K}\tilde t)o(p^s-1)=op^{\tilde s}(p^{\hat s}-1).
\]
  Consequently, we must have $(p^s-1) | (p^{\hat s}-1)$ and $s |\hat s$.  Further from \eqref{hat powert} we also have that $\ord_q\hat s=s_0+n_q$. .  Therefore, if {\tt T} is true for some $f \in M$ then there exists infinitely many $\hat s$ with $\ord_q\hat s=n_q+s_0$ and such that \eqref{eq:xt2} and \eqref{eq:normsit2} hold.  Thus by Lemma \ref{le:inK}, we conclude that $f \in K$.  
  
  Conversely, suppose $f \in \F_{p^{z}}[t] \subset \G_p[t]$ for positive integer $z$.  Note that by definition of $n_q$ and $M$, in this case $\ord_qz \leq n_q$.   Let $v \in M$ be given. If 
\[
\forall s \in \Z_{\geq 0}: v \not =t^{p^s} \mbox{ with } s \equiv 1\mod q,
\]
 then \eqref{clause :1t}  is true and {\tt T} is true.  So we can assume that 
\[
 \exists s \equiv 1 \mod q: v=t^{p^s} 
\]
    Let $\hat s$ be a multiple of $szs_2$, and  note that $\ord_q \hat s\leq n_q+s_0$.  Let $\tilde s=s_1$ so that
\[
o(p^s-1)|p^{\tilde s}(p^{\hat s}-1).
\]
Now let 
\[
w=\frac{p^{\tilde s}(p^{\hat s}-1)}{o(p^s-1)}, \tilde t =t^w,
\]
 and observe that \eqref{compare powerst2} now holds.   
Finally, since $z |\hat s$, for some $i$ we have that \eqref{eq:normsit2} holds.  Any rational function in $\G_p(t)$ can of course be written as a ratio of two polynomials and we can use a basis of $K/\G_p(t)$ as parameters to generate the rest of the field.

\end{proof}
\subsection{Using extensions of degree $p$}
The proof in this case proceeds in an analogous fashion, except that we will need to substitute an extension  of degree which is a power of $p$ for an extension of degree which is a power of $q$.  Also the norm equations \eqref{eq:normsit2} assuring integrality at a set of primes have to be modified along the lines of \eqref{eq:normsi1p}. \\

Before we state the main theorem of this section, we need the following well-known lemma whose proof we omit.
\begin{lemma}
\label{le:perfect}
Let $H$ be a perfect field of characteristic $p>0$, let $t$ be transcendental over $H$, and let $M$ be a finite extension of $H(t)$.  In this case the extension $M/H(t)$ is separable if and only if $t$ is not a $p$-th power in $M$.
\end{lemma}

Summarizing the results of this section, we have the following theorem.

\begin{theorem}
\label{thm:general}
Let $H$ be any field of positive characteristic $p$ not containing the algebraic closure of a finite field.  Let $t$ be transcendental over $H$ and let $M$ be a finite extension of $H(t)$.  Further, let $\G_p$ be the algebraic closure of a finite field in $H$ and let $\F_{p^s}$ be any finite field contained in $H$.  In this case $\G_p[t]$ ans $\F_{p^s}[t]$ are first-order definable (with parameters) over $M$.
\end{theorem}
\begin{proof}
Before we proceed with the proof we note that it would be enough to show that the proposition holds for a finite extension of the given field (as we can rewrite all the equations to have variables range in the given field and to have all the coefficients in the given field). 

We start with a definition of $\G_p[t]$.  The only point we need to address before applying Proposition \ref{prop:definable} is the requirement in Notation and Assumptions \ref{not:transcendental} that $M/H(t)$ is a Galois extension.    

Let $\beta_1,\ldots, \beta_k$ be  generators of $M$ over $H(t)$, or in other words let $M=H(t,\beta_1,\ldots, \beta_k)$.  Let $H_{ins}$ be the inseparable closure of $H$,  let $t=w^{p^m}$ for some $m \in \Z_{\geq 0}$  and for some $w \in H_{ins}M$ such that $w$ is not a $p$-th power in $H_{ins}M$.  Note that an element $x \in M$ is a polynomial in $t$ with coefficients in $\G_p$  if and only if it is a $p^m$-th power of a polynomial in $w$ with coefficients in $\G_p$.  Thus it is enough to define the polynomial ring in $w$ with coefficients in $\G_p$.   

Since $w$ is not a $p$-th power, by Lemma \ref{le:perfect}, the extension $H_{ins}M/H_{ins}(w)$ is separable. The extension is also finite, since $H_{ins}M=H_{ins}(w, \beta_1,\ldots,\beta_k)$.  Thus, this extension is simple.  Let $\alpha$ be a generator and let $R(T)=A_0(w)+A_1(w)T +\ldots +T^r$ be the monic irreducible polynomial of $\alpha$ over $H_{ins}(w)$.  Note that by assumption on $\alpha$, it is the case that $R(T)$ does not have multiple roots in the algebraic closure of $H(w)$.  Write $\beta_i=\sum_{j=0}^{r-1}g_{i,j}(w)\alpha^j, g_{i,j} \in H_{ins}(w)$.  Now let $\bar H$ be a finite extension of $H$ containing all the coefficients of $A_0(w), \ldots, A_{r-1}(w), g_{1,0}(w),\ldots, g_{k,r-1}(w)$ and consider the extension $\bar H(\alpha, w)/\bar H(w)$.  This extension is finite and separable, while $\bar H(\alpha,w)$ is a finite extension of $M$.  Finally we can take the Galois closure of $\bar H(\alpha, w)$ over $\bar H(w)$ to satisfy the the requirement in Notation and Assumptions \ref{not:transcendental}.

Finally, to obtain a definition of $\F_{p^s}[t]$, we can define $\G_p[t]$ first and consider elements inside $\G_p[t]$ satisfying $f^{p^{u}}-f\equiv 0 \mod t^{p^{u}}-t$ and $f^{p^{u+s}}-f\equiv 0 \mod t^{p^{u+s}}-t$ for all $u \in \Z_{>0}$.

\end{proof}
\section{Appendix: Rumley's Formula}
In this appendix we estimate a lower bound on the number of universal quantifiers necessary to rewrite R. Rumely's formulas in the prenex normal form.  We rewrite the formula defining  what he called $O_x$, the integral closure of the polynomial ring in $x$ in the global field in question.  To simplify the matters, we consider the case of the characteristic greater than 2, so that in Rumeley's notation we require the case of $l=2$ only.  We therefore suppress this index in the calculations below and rewrite all the formulas part of the way towards the prenex normal form, far enough to estimate the number of universal quantifiers required.  We start with the quadratic norm equations:
\[
N_2(b,a_0,a_1)=N_2(b,\vec{a})=a_0^2-a_1^2b=N(b,a_0,a_1),
\]
and continue through the formulas defining valuation rings:
\[
R(t;c;d) \Longleftrightarrow \exists \vec{a}_1\exists \vec{a}_2\exists \vec{a}_3\exists w (w=N(d,\vec{a}_1) \land cw=N(cd,\vec{a}_2) \land t=N(w,\vec{a}_3)),
\]

\[
S(x, c_1,d_1,c_2,d_2) \Longleftrightarrow \exists t_1\exists t_2 (1+c_1x^l=t_1t_2\land R(t_1,c_1,d_1) \land R(t_2,c_2,d_2)) \Longleftrightarrow
\]
\[
\exists t_1\exists t_2 (1+c_1x^l=t_1t_2\land \exists \vec{a}_{1,1}\exists \vec{a}_{1,2}\exists \vec{a}_{1,3}\exists w_1 (w_1=N(d_1,\vec{a}_{1,1}) \land c_1w_1=N(c_1d_1,\vec{a}_{1,2}) \land t_1=N(w_1,\vec{a}_{1,3})) 
\]
\[
\land \exists \vec{a}_{2,1}\exists \vec{a}_{2,2}\exists \vec{a}_{2,3}\exists w_2 (w_2=N(d_2,\vec{a}_{2,1}) \land c_2w_2=N(c_2d_2,\vec{a}_{2,2}) \land t_2=N(w_2,\vec{a}_{2,3}))) \Longleftrightarrow
\]
\[
\exists t_1\exists t_2 \exists \vec{a}_{1,1}\exists \vec{a}_{1,2}\exists \vec{a}_{1,3}\exists w_1  \exists \vec{a}_{2,1}\exists \vec{a}_{2,2}\exists \vec{a}_{2,3}\exists w_2:
\]
\[
1+c_1x^l=t_1t_2\land w_1=N(d_1,\vec{a}_{1,1}) \land c_1w_1=N(c_1d_1,\vec{a}_{1,2})\land t_1=N(w_1,\vec{a}_{1,3}) \land
\]
\[
 w_2=N(d_2,\vec{a}_{2,1}) \land c_2w_2=N(c_2d_2,\vec{a}_{2,2}) \land t_2=N(w_2,\vec{a}_{2,3})
\]
\[
\Longleftrightarrow  \exists t_1\exists t_2 \exists \vec{a}_{1,1}\exists \vec{a}_{1,2}\exists \vec{a}_{1,3}\exists w_1  \exists \vec{a}_{2,1}\exists \vec{a}_{2,2}\exists \vec{a}_{2,3}\exists w_2:
\]
\[
P(x,c_1,d_1,c_2,d_2,  t_1,t_2,\vec{a}_{1,1},\vec{a}_{1,2},\vec{a}_{1,3}, w_1 , \vec{a}_{2,1},\vec{a}_{2,2}, \vec{a}_{2,3},w_2),
\]
where $P(\ldots)$ is a system of polynomial equations.
Let 
\[
\vec{u} =( t_1,t_2,\vec{a}_{1,1},\vec{a}_{1,2},\vec{a}_{1,3}, w_1 , \vec{a}_{2,1},\vec{a}_{2,2}, \vec{a}_{2,3},w_2)
\]
 and let 
\[
S(x, \vec c)=S(x,c_1,c_2,c_3,c_4)\Longleftrightarrow \exists \vec{u} P(x,c_1,c_2,c_3,c_4, \vec u).
\]
Continuing to follow Rumley's formulas we get
\[
V(x,\vec c) \Longleftrightarrow 
\]
\[
[\forall y\forall z(\lnot S(y, \vec c) \lor \lnot S(z, \vec c) \lor (S(-y, \vec c)\land S(y+z, \vec c)\land S(yz, \vec c) \land ((y=0) \lor S(y, \vec c) \lor S(1/y, \vec c)))] 
\]
\[
\Rightarrow S(x, \vec c)
\]
\[
\Longleftrightarrow
\]
\[
 [\exists y \exists z (S(y, \vec c) \land S(z, \vec c) \land (\lnot S(-y, \vec c) \lor \lnot S(y+z, \vec c) \lor \lnot S(yz, \vec c) \lor  ((y \not =0) \land \lnot S(y, \vec c) \land \lnot S(1/y, \vec c))))] 
 \]
 \[
 \lor  S(x, \vec c)
\]

Now  $O_x$ is defined by the formula  $[\forall \vec c(V(x, \vec{c} ) \Rightarrow  V(t, \vec{c} ) )] \Longleftrightarrow[\forall \vec c(\lnot V(x, \vec c) \lor V(t, \vec{c} ) )]\Longleftrightarrow  $
\[
\forall \vec c
\]
\[
[\forall y\forall z(\lnot S(y,\vec c) \lor \lnot S(z, \vec c) \lor (S(-y,\vec c)\land S(y+z,\vec c)\land S(yz, \vec c)) \land ((y=0) \lor S(y, \vec c) )\lor S(1/y,\vec c)))
\]
\[
\land \lnot S(x,\vec c))
\]
\[
 \lor 
\]
\[
\exists y \exists z (S(y,\vec c) \land S(z,\vec c) \land (\lnot S(-y,\vec c)
\]
\[
 \lor \lnot S(y+z,\vec c) \lor \lnot S(yz,\vec c) \lor ((y \not =0) \land \lnot S(y, \vec c) \land \lnot S(1/y, \vec c))) \lor  S(t, \vec c)]
\]
Above and beyond the variables $c_1,\ldots,c_4, y,z$ appearing in the range of a universal quantifier, we  have 10 variables appearing in the range of a universal quantifier in the negation of $S(\ldots, \vec c)$.  Thus, even assuming we can reuse these variables for all negations of $S(\ldots, \vec c)$  we still have at least  16 variables in the range of a universal quantifier.  The definition of the polynomial ring has the definition of $O_x$ as a part of a conjunction.  So it will require at least as many quantifers.

\end{document}